          \documentclass{cmslatex}
\usepackage[paperwidth=7in, paperheight=10in, margin=.875in]{geometry}
\usepackage{amsmath}
\usepackage{setspace,graphicx}     %!  setspace is used to control linepacing
\usepackage{amsfonts}
\usepackage{subfig}
\usepackage{latexsym}
\usepackage{upgreek}
\usepackage{amssymb}

\usepackage{amsthm}
\usepackage{setspace}
\usepackage{fancyhdr}
\usepackage{mathrsfs}
\usepackage{framed}
\usepackage{bm}
\usepackage{perpage} 
\usepackage{bbm}
\usepackage{dsfont}
\usepackage{comment}
\usepackage{color}
\definecolor{darkgreen}{rgb}{0,0.7,0}

\definecolor{darkblue}{rgb}{0,0,0.7}

\renewcommand{\div}{\text{div}}

\newcommand{\R}{\mathbb{R}}

\newcommand{\N}{\mathbb{N}}

\newcommand{\Sp}{\mathbb{S}_+(\R^d)}
\newcommand{\h}{\mathcal{H}}

\newcommand{\F}{\mathcal{F}}

\newcommand{\dhell}{D_{\mbox {\tiny{\rm Hell}}}}
\newcommand{\Las}{L_{A}^s}
\newcommand{\Span}{span}

\newcommand{\E}{\mathbb{E}}

\newcommand{\envelope}{(\raisebox{-.5pt}{\scalebox{1.45}{\Letter}}\kern-1.7pt)}

\newcommand{\veps}{\varepsilon}
\newcommand{\Lavg}{L^2_{\mbox {\tiny{\rm avg}}}(D)}
\newcommand{\Lpavg}{L^p_{\mbox {\tiny{\rm avg}}}(D)}
\newcommand{\Havg}{H^1_{\mbox {\tiny{\rm avg}}}(D)}
\newcommand{\Hsavg}{H^s_{\mbox {\tiny{\rm avg}}}(D)}
\newcommand{\Honeavg}{H^1_{\mbox {\tiny{\rm avg}}}}
\newcommand{\Hscalavg}{{\cal H}^s_{\mbox {\tiny{\rm avg}}}(D)}
\newcommand{\Hmscalavg}{{\cal H}^{-s}_{\mbox {\tiny{\rm avg}}}(D)}

\renewcommand{\theequation}{\arabic{section}.\arabic{equation}}

            \newtheorem{thm}{Theorem}[section]
          \newtheorem{prop}[thm]{Proposition}
          \newtheorem{lem}[thm]{Lemma}
          
          \newtheorem{remark}[thm]{Remark}
		\newtheorem{assumption}[thm]{Assumption}
			\newtheorem{setting}[thm]{Setting}

\begin{document}
          \title{The Bayesian Formulation and Well-Posedness of Fractional Elliptic Inverse Problems}

          %For each author, make a block with the following four macros:

          \author{{Nicol\'{a}s Garc\'{i}a Trillos }
         \and{Daniel Sanz-Alonso}}

          %Use \thanks statements for acknowedgements of grants and
          %support. They will appear below all the authors' addresses, so be
          %specific about which author is thanking whom:

          %\thanks{}

          % Use the standard latex environments for theorems, etc. Here is one
          % possible method of declaring them: It numbers all results by the
          % section, and uses a common numbering system for the different
          % environmentts.

          %\date{Received date / Revised version date}

         \pagestyle{myheadings} \markboth{The Bayesian Formulation and Well-Posedness of Fractional Elliptic Inverse Problems}{N. Garc\'{i}a Trillos, and D. Sanz-Alonso} \maketitle

          \begin{abstract}
        We study the inverse problem of recovering the order and the diffusion coefficient of an elliptic fractional partial differential equation from a finite number of noisy observations of the solution. We work in a Bayesian framework and show conditions under which the posterior distribution is given by a change of measure from the prior. Moreover, we show well-posedness of the inverse problem, in the sense that small perturbations of the observed solution lead to small Hellinger perturbations of the associated posterior measures. We thus provide a mathematical foundation to the Bayesian learning of the order ---and other inputs--- of fractional models.
          \end{abstract}

 %\begin{AMS} to do
%\end{AMS}
   \section{Introduction}\label{intro}
   \subsection{Aim and Relevance}
         The great promise of nonlocal models described by fractional partial differential equations (FPDEs) has been largely confirmed in applications as varied as groundwater flow, finance, materials science, and mathematical biology. A common scenario is that  fractional order models are well established, and yet the correct order is hard to determine as it may depend on specific features of the problem. This is for instance the case in the modeling of viscoelastic materials \cite{bagley1983fractional}, where in practice the order and parameters of the models are determined by laboratory experiments that often contain non-negligible amounts of uncertainty and errors \cite{sasso2011application}. The aim of this paper is to provide a probabilistic framework for a fractional elliptic equation. This framework acknowledges any uncertainty in the order and the diffusion coefficient of the equation, and allows to reduce said uncertainty by the use of partial and noisy measurements of the output solution. In this way we extend the Bayesian formulation of elliptic inverse problems proposed in \cite{dashti2011uncertainty}  to fractional order models, and introduce ---to the best of our knowledge--- the first Bayesian approach to learning the order of a FPDE.

      	Other than their applied importance in modelling, there is a further motivation for the study of fractional equations, and specifically their inversion. Indeed, the solution to classical integer order PDEs  in physical space can some times be recovered from a FPDE in the lower-dimensional boundary of the physical domain. Output measurements are often taken in the boundary, and hence relate naturally to the fractional equation. With the Bayesian approach proposed in this paper, uncertainty  in the solution in the full physical domain ---after boundary measurements are taken--- could be characterized as follows: i) describe the uncertainty in the inputs of the FPDE; ii) reduce the uncertainty by use of boundary measurements; iii) propagate the remaining uncertainty in the inputs to characterize the uncertainty in the solution to the FPDE; and iv) use the mapping from boundary to physical domain to characterize the uncertainty in the full solution. An example is given by the full 3D quasi-geostrophic equations, whose streamlines can be recovered from the 2D surface quasi-geostrophic equations (which contains a fractional diffusive term) by solving an elliptic PDE in 3D. Data assimilation using boundary measurements for this system has been studied in a sequential context \cite{jolly2016data}. An example that is more closely related to the framework of this paper is given by the thin obstacle problem in the introduction of \cite{caffarelli2016fractional}.

         \subsection{Framework and Main Results}           
     In order to describe the inverse problem of interest, we now introduce a working definition of the FPDE that will constitute our forward model.  A more detailed mathematical account will be given in Section \ref{sec:forward}.   We work in a bounded Lipschitz domain $D\subset \R^d,$ and let $L_A := -\div (A(x) \nabla_x).$ For $0<s< 1$ we let $L_A^s$ denote a fractional power of the elliptic operator $L_A$ (see equation \eqref{def:fractional}, Section \ref{sec:forward}). We then consider the Neumann problem
          \begin{align} \label{fractionalpde} 
          \begin{split}
          \begin{cases}
          \Las p &= f, \quad \text{in} \,\,\, D, \\
          \partial_A p &= 0, \quad \text{on}  \,\,\,\partial D,
          \end{cases}
          \end{split}
          \end{align}
          where $\partial_A p:= A(x) \nabla p \cdot \nu,$ and $\nu$ is the exterior unit normal to $\partial D.$ Extensions to other boundary conditions are possible. The right-hand side $f$ is assumed to be known, and conditions on its regularity will be given.  The inputs $s$ and $A$ are assumed to contain non-negligible uncertainty. We suppose, however, that the diffusion coefficient $A= A(x)$ is known to be symmetric and strictly elliptic. The later means that there are positive constants $\lambda_A$ and $\Lambda_A$ such that, for almost every $x\in D,$ 
   \begin{equation}
    \lambda_A I_d \leq  A(x) \leq \Lambda_A I_d,  
   \label{UnifElliptic}
   \end{equation}
   where $I_d$ is the $d\times d$ identity matrix, and $M_1\ge M_2$ if $M_1-M_2$ is positive semi-definite.   The constants $\lambda_A$ and $\Lambda_A$ can be recovered sharply in terms of the minimum and maximum eigenvalues of the $A(x)$ over $D$. For such optimal $\lambda_A$ and $\Lambda_A$ we refer to $\Lambda_A/\lambda_A$ as the {\em ellipticity} of $A$. 
            Under mild assumptions equation \eqref{fractionalpde} has a unique solution $p=p_{s,A}$. The {\em forward map} is then defined as the map ${\cal F}: X\to Z$ from inputs $(s,A)\in X$  to the solution $p_{s,A}\in Z.$ We consider different choices of input parameter space $X$ and corresponding space of outputs $Z$, see Settings \ref{setting1} and \ref{setting2} below.  We investigate the inverse problem of learning the inputs $(s,A)\in X$  from a finite dimensional vector $y\in \R^m$ of partial and noisy measurements of the output solution $p\in Z.$ More precisely, we assume the additive Gaussian observation model         
                    \begin{equation}
          y = {\cal G}\bigl((s,A)\bigr) + \eta,
          \end{equation}
where ${\cal G}$ is the composition of the forward map ${\cal F}:X\to Z$ with a bounded linear functional ${\cal O}:Z \to \R^m$ representing an {\em observation map}, and $\eta$ is a vector of measurement errors that we assume to be centered and Gaussian with known positive definite covariance $\Gamma$,  $\eta \sim N(0,\Gamma).$                
               
              We follow the Bayesian approach to inverse problems \cite{kaipio2006statistical}, \cite{AS10} and put a prior distribution $\mu_0$ on the inputs $(s,A)\in X,$ aiming to capture both the uncertainty and the available knowledge about the inputs. The prior is then conditioned on the observed data $y$ to produce ---via Bayes' rule--- a {\em posterior} distribution $\mu^y$ on the space $X$ of input parameters. Note, however, that application of Bayes' rule in this setting requires careful justification since our space of parameters is  infinite dimensional. We will provide such justification here in two different settings. That is the content of our first main result:  
              
              \begin{thm}\label{thm:maintheorem1}
Under the conditions of Setting \ref{setting1} or Setting \ref{setting2} below, the forward map ${\cal F}: X\to Z$ is continuous.
Therefore, if the prior $\mu_0$ is any measure with $\mu_0(X) = 1,$ then the Bayesian inverse problem of recovering inputs  $u:=(s,A) \in X$ of the FPDE \eqref{fractionalpde} from data 
$$y = {\cal G}(u) + \eta, \quad \eta \sim N(0,\Gamma),$$
is well formulated: the posterior $\mu^y$ is well defined in $X$ and it is absolutely continuous with respect to $\mu_0.$ Moreover, the Radon-Nikodym derivative is given by         \begin{equation}
\frac{d\mu^y}{d\mu_0}(u) = \frac{1}{Z}\exp \Bigl(- \frac{1}{2}|y-{\cal G}(u)|_{\Gamma}^2, \Bigr), \quad Z= \int_X \exp \Bigl(- \frac{1}{2}|y-{\cal G}(u)|_{\Gamma}^2 \Bigr)\, d\mu_0(u),
\end{equation}
where $|\cdot|_\Gamma = |\Gamma^{-1/2}\cdot|,$ and $|\cdot|$ is the Euclidean norm in $\R^m.$
\end{thm}
The proof is given in Section \ref{sec:formulation}. We remark that measurability of ${\cal G}$ would suffice for the above result to hold. Measurability of ${\cal G}$ is implied by the shown continuity of ${\cal F}$ and the assumed continuity of ${\cal O}.$

Our second main result concerns well-posedness of the Bayesian inverse problem, in the sense that small perturbations in the data lead to small Hellinger perturbations of the corresponding posterior measures. We recall that the Hellinger distance between two probability measures $\mu, \mu'$ (defined in the same measurable space) is given by
  \begin{equation}
   \dhell(\mu,\mu')^2 = \frac12 \int\biggl(\sqrt{\frac{d\mu}{d\nu}} - \sqrt{\frac{d\mu'}{d\nu}} \,\,\,\biggr)^2 \, d\nu,
  \end{equation}
  where $\nu$ is {\em any} reference probability measure with respect to which both $\mu$ and $\mu'$ are absolutely continuous (e.g. $\frac12 \mu+ \frac12 \mu'$).
  We then have:

\begin{thm}\label{thm:maintheorem2}
Suppose Setting \ref{setting1} or Setting \ref{setting2} below are in place, and that the prior $\mu_0$ satisfies Assumption \ref{assumptionprior}. Then there is $C=C(r)$ such that, for all $y_1, y_2 \in \R^m$ with $|y_1|, |y_2|\le r,$
          \begin{equation}\label{stabilityhellinger}
     \dhell(\mu^{y_1}, \mu^{y_2}) \le C |y_1 - y_2|.
          \end{equation}
\end{thm}

The proof can be found in Section \ref{sec:hellinger}.  Stability of the Hellinger distance guarantees stability of posterior expectations under perturbations in the data. Precisely, \eqref{stabilityhellinger} implies that for $h\in L^2_{\mu^y}(X)\cap L^2_{\mu^{y'}}(X)$  there is $C,$ depending only on $r$ and on the expectation of $h^2$ under $\mu^y$ and $\mu^{y'},$ such that 
 $$\Bigl|\E^{\mu^y} h - \E^{\mu^{y'}} h\Bigr| \le C |y-y'|.$$
 
\subsection{Scope and Novelty} 
 
Existence and well-posedness results similar to Theorems \ref{thm:maintheorem1} and \ref{thm:maintheorem2} have been established for a number of Bayesian inverse problems, see e.g. \cite{AS10}, \cite{dashti2011uncertainty}. However, the proofs of our main results require new techniques that rely on state-of-the art (F)PDE regularity theory. We now briefly review some of the novel features in the scope and analysis of the Bayesian inverse problem studied in this paper: 

\begin{enumerate}
\item Bayesian learning of the order of the model ---and potentially of spatially-variable order models--- is bound to find applications in finance, material science, the geophysical sciences, and beyond. Our results build on the recently developed theory of Bayesian inverse problems in function space \cite{AS10}. We provide a brief review in Section \ref{sec:bip}.
\item We show (Subsection \ref{ssec:spectralapproach}) continuity of the forward map using spectral theory of self-adjoint, compact operators. This proof relies on the spectral definition of the fractional operator $\Las$, described in Subsection \ref{ssec:spectralproperties}, and allows, at most, for right-hand side $f\in L^2(D).$
\item We make use of the extension problem for elliptic FPDEs, reviewed in Subsection \ref{ssec:extensionproblem}. This powerful idea allows to study elliptic FPDEs by means of an associated elliptic PDE in higher dimensional space with degenerate diffusion coefficient. The extension problem for the fractional Laplacian was introduced in \cite{caffarelli2007extension}, and then generalized in \cite{stinga2010fractional}, \cite{stinga2010extension} to more general second order elliptic operators, such as the ones considered in this paper. We provide a second proof of continuity of the forward map using the extension problem in Subsection \ref{ssec:extensionapproach}. This approach allows for source $f\in H^{-s}$, at least for certain priors. 
\item The proof of Theorem \ref{thm:maintheorem1}, and specially that of Theorem \ref{thm:maintheorem2}, requires careful analysis of how different constants appearing in regularity estimates depend on the ellipticity of the base elliptic operator $L_A$. In particular, and as part of our analysis, we study the effect of the ellipticity on different fractional Sobolev norms (Remark \ref{remark:boundnorms}) and on the Cacciopoli estimates in \cite{caffarelli2016fractional}. The later is necessary in order for the theory to cover log-normal type priors ---widely used in applications--- for the diffusion coefficient. 
\end{enumerate}

\subsection{Literature}

	We conclude this introduction by relating our work to the literature.  An extensive review of fractional dynamics, their applications, and their connection to stochastic processes is \cite{klafter2005anomalous}. The interplay between fractional diffusion and stochastic processes sheds light into their key applied relevance: the Feynman-Kac formula for general $\alpha$-stable L\'evy processes \cite{bertoin1998levy}, \cite{applebaum2009levy}  ---widely used, for instance, in finance--- is a fractional Laplacian diffusion \cite{bochner1949diffusion}  (with integer order for Brownian motion). Fractional derivatives have been used to model groundwater flow \cite{atangana2013use}, and a deep analysis of the fractional porous medium equation is given in \cite{de2011fractional}. The regularity theory for the Neumann problem \eqref{fractionalpde} ---as well as the Dirichlet problem--- has been thoroughly studied in \cite{caffarelli2016fractional}. Two important tools in the analysis of elliptic FPDEs are the extension problem, by which the analysis can be reduced to that of an elliptic PDE with degenerate diffusion coefficient \cite{caffarelli2007extension}, \cite{stinga2010extension} \cite{stinga2010fractional}, and the use of characterizations of the fractional Laplacian based on the heat semigroup or Poisson kernels \cite{caffarelli2016fractional}.  On the computational side, finite element methods for fractional problems have also been studied via the extension argument \cite{nochetto2015pde} ---see also \cite{acosta2015fractional}. The inverse problem considered here could be amenable to classical regularization approaches \cite{engl1996regularization}, and there has been recent interest in inverse problems for related fractional models \cite{jin2015tutorial}, \cite{zhang2015undetermined}.  The Bayesian formulation that we adopt has, however,  two main appealing features, see e.g. \cite{AS10}. First, the prior provides a natural way (with a clear probabilistic interpretation) of regularizing the otherwise underdetermined inverse problem. Second, the solution to the Bayesian inverse problem, i.e. the posterior measure, contains information on the remaining uncertainty in the inputs after the observations have been assimilated. A precise understanding of the uncertainty in the inputs is key in order to characterize the uncertainty in the solution to  \eqref{fractionalpde}. In this regard, the numerical propagation of uncertainty through differential models is an active area of research \cite{xiu2010numerical}, \cite{xiu2002wiener}.   A textbook on  the Bayesian approach to inverse problems is \cite{kaipio2006statistical}. The formulation was extended to function space settings ---such as the one considered in this paper--- in \cite{AS10}. The infinite dimensional elliptic inverse problem was studied in \cite{dashti2011uncertainty}, and posterior consistency was established in \cite{SV13}. The derivation of posterior consistency results in the fractional setting will be the subject of future work, as will be the investigation  of spatially-varying order models \cite{zayernouri2015fractional} and priors \cite{calvetti2015variable}. We also intend to study the specific computational challenges of the Bayesian inverse problem  arising from the fractional forward model.
          
         \paragraph{Outline}  Section \ref{sec:bip} reviews the Bayesian approach to inverse problems in function space. 
     Section \ref{sec:forward} introduces the mathematical formulation of the  forward model \eqref{fractionalpde}.      In Section \ref{sec:formulation} we show continuity of the forward map, thereby proving Theorem \ref{thm:maintheorem1}. Section \ref{sec:hellinger} contains the proof of Theorem \ref{thm:maintheorem2}. A simple example is given in Section \ref{sec:example}. We close in Section \ref{sec:conclusion}. The proofs of some auxiliary results are brought together in an appendix.
          \paragraph{Notation}      We let  $\Sp$ be the space of $d\times d$ real positive definite matrices. Function spaces of zero-mean functions will be denoted with a subscript ``{\text avg}". For instance, $\Lavg$ will denote the space of functions $h\in L^2(D)$ with $\int_D h(x)\, dx = 0.$
          
                    \section{Bayesian Inverse Problems}\label{sec:bip}
          Let $X$  and $Z$ be two separable Banach spaces. $X$ will represent the space of input parameters for \eqref{fractionalpde}, and $Z$ will represent the space of corresponding output solutions. Let ${\cal F}:X\to Z$ be the {\em forward map} from inputs to outputs, and let ${\cal O}:Z \to \R^m$ be the {\em observation map} from outputs to data. Suppose that $\cal F$ and $\cal O$ are Borel measurable maps and denote ${\cal G}:= {\cal O} \circ {\cal F}.$ We consider the inverse problem 
          \begin{equation}\label{inverse problem}
          y = {\cal G}(u) + \eta,
          \end{equation}
          where the aim is to recover the input $u$ from data $y.$ We assume that $\eta \sim N(0,\Gamma)$ for given positive definite $\Gamma.$  We follow the Bayesian approach and put a prior $\mu_0$ on the unknown $u$. The conditional law of $u$ given $y$ is known as the posterior measure, and will be denoted $\mu^y.$ The following two propositions are an immediate consequence of the theory of inverse problems in function space introduced in \cite{AS10}, and further developed in \cite{DS15}. We will use them to show our main results Theorem \ref{thm:maintheorem1} and \ref{thm:maintheorem2}.

          \begin{prop}[Posterior definition]\label{wellposed}
          Suppose that the map ${\cal G}: X\to \R^m$ is measurable and that $\mu_0(X) = 1.$ Then the posterior distribution $\mu^y$ is absolutely continuous with respect to $\mu_0,$ and the Radon-Nikodym derivative is given by
          \begin{equation}\label{radonnikodym}
          \frac{d\mu^y}{d\mu_0}(u) = \frac{1}{Z}\exp \Bigl(-\frac{1}{2}|y-{\cal G}(u)|_{\Gamma}^2, \Bigr), \quad Z= \int_X \exp \Bigl(-\frac{1}{2}|y-{\cal G}(u)|_{\Gamma}^2, \Bigr)\, d\mu_0(u).
          \end{equation}
          \end{prop}

%The typical difficulty found when trying to use the above Proposition is in establishing the measurability of the map $\mathcal{G}$. We notice that one direct way to obtain the measurability of $\mathcal{G}$ is to establish its continuity. In Section \ref{sec:formulation} we establish the continuity of the map $\mathcal{G}$ that is of interest to our problem.

        \begin{prop}[Hellinger continuity] \label{propositionhell}
          Assume that $\mu_0(X) = 1,$ and that ${\cal G}\in L^2_{\mu_0}(X).$
          Then there is $C=C(r)$ such that, for all $y_1, y_2 \in \R^m$ with $|y_1|, |y_2|\le r,$
          \begin{equation*}
     \dhell(\mu^{y_1}, \mu^{y_2}) \le C |y_1 - y_2|.
          \end{equation*}
       
          \label{PropHellinger}
          \end{prop}
  
%The hypotheses in the above proposition essentially say that pathological $\mathcal{G}(u)$ occur with small probability according to the prior $\mu_0$. In other words, the prior gives little weight to those parameters $u$ that produce big values for $\lvert \mathcal{G}(u) \rvert$.  In our problem where $u=(s,A)$ we can relate $\lvert \mathcal{G}(u)\rvert$ to the reciprocal of the ellipticity constant $c_A$ of $A$. So, roughly speaking, we will assume that the prior $\mu_0$ gives little weight to those $A$s for which the ellipticity condition degenerates. We give more details and provide the precise assumptions in Section \ref{sec:hellinger}.

\section{Forward Model}  \label{sec:forward}
          In this section we give the mathematical formulation of the forward model \eqref{fractionalpde}. It is important to note that there is no canonical way to define fractional powers of the base elliptic operator $L_A = -\div (A(x) \nabla_x)$ in bounded domains. In this paper we adopt the spectral definition --see equation \eqref{def:fractional} below. The associated fractional elliptic problem  has been recently studied both analytically   \cite{caffarelli2016fractional} and computationally \cite{nochetto2015pde}. We refer to \cite{acosta2015fractional} for further discussion on different definitions of fractional Laplacians. 
   
 \subsection{Basic Formulation and Spectral Considerations}  \label{ssec:spectralproperties} 
 
The diffusion coefficient $A(x)\in \R^{d\times d}$ of $L_A$ will be assumed to be symmetric, bounded, measurable, and to satisfy a  uniform elliptic condition, see  \eqref{UnifElliptic}.
   Since we are considering the Neumann problem, we restrict the domain of  $L_A$ to  $\Havg,$ and note that $\Lavg$ admits an orthonormal basis of eigenfunctions of $L_A,$ $\psi_k\in \Havg,$ $k\ge 0,$   with corresponding eigenvalues $0<\lambda_1 \le \lambda_2 \nearrow\infty$. This allows us to define, for $0<s < 1$ and $p(x) = \sum_{k=1}^\infty p_k \psi_k(x),$ 
   \begin{equation}\label{def:fractional}
   \Las p(x) = \sum_{k=1}^\infty \lambda_k^s p_k \psi_k(x).
   \end{equation}
The domain of $\Las$ is the space $\Hscalavg$ of functions $p\in \Lavg$ with
\begin{equation}
\|p\|_{\h_{A}^s}^2:=\sum_{k=1}^\infty \lambda_k^s \langle p, \psi_k\rangle_{L^2}^2 <\infty.
\end{equation}
This space has Hilbert structure when endowed with the inner product 
$$\langle p,q \rangle_{\h_A^s} := \sum_{k=1}^\infty \lambda_k^s p_k q_k,$$
where $q = \sum_{k=1}^\infty q_k \psi_k \in \Hscalavg.$
The space $\Hscalavg$ does not depend on $A$. Indeed $\Hsavg$ can be equivalently defined as zero-mean functions in the closure of $C^\infty(D)$ with respect to the norm $\|\cdot \|^2_{H^s} := \|\cdot \|_{L^2}^2 + [\,\cdot\,]^2_{H^s},$
where $$[\,p\,]_{H^s}:= \int_D \int_D \frac{\bigl( p(x)-p(z)\bigr)^2}{|x-z|^{d+2s}} \,dx\, dz,$$
see \cite{caffarelli2016fractional}. Indeed, the analysis in Subsection \ref{ssec:non-smooth} ---see Remark \ref{remark:boundnorms}---   shows that there is $C_s$ independent of $A$ such that $[p]_{H^s} \le C_s \| p \|_{\h_A^s}.$ This will be used in Subsection \ref{ssec:smoothcase} in to order to formulate the Bayesian inverse problem in the case of smooth $A.$ 

     Any functional $f\in \Hmscalavg$ acting on $\Hscalavg$ can be written as $f= \sum_{k=1}^\infty f_k \psi_k$, where $\sum_{k=1}\lambda_k^{-s} f_k^2 <\infty.$ For any such $f$ there is a unique solution $p=p_{s,A}\in \Hscalavg$ to \eqref{fractionalpde} given by
     \begin{equation*}
     p_{s,A}= \sum_{k=1}^\infty \lambda_k^{-s} f_k \psi_k.
     \end{equation*}
In the extreme case $s=0$,  if $f\in \Lavg$ then \eqref{fractionalpde} has a unique solution $p\in\Lavg$. In Section \ref{sec:formulation} we study the spectral properties of the operator $L_A^{-1}$ that maps $g\in \Lavg$ to the solution $p\in \Lavg$ to 
 \begin{align} 
          \begin{split}
          \begin{cases}
          L_A p &= g, \quad \text{in} \,\,\, D, \\
          \partial_A p &= 0, \quad \text{on}  \,\,\,\partial D.
          \end{cases}
          \end{split}
          \end{align}
    The previous paragraph implies that $L_A^{-1}$ is well defined. Moreover $L_A^{-1}$ is continuous, compact, and self-adjoint with respect to the usual $L^2$-inner product.                    
          
          \subsection{The Extension Problem}\label{ssec:extensionproblem}
     In this subsection we introduce the extension problem that will be used to show continuity of the forward map in Subsection \ref{ssec:extensionapproach}.
         
          For uniformly elliptic $A$ and $s \in (0,1)$ let $a:= 1-2s$, and let
          \begin{equation}\label{definitionB}
          B(x) := \left( \begin{matrix} A(x) & 0 \\ 0 & 1 \end{matrix}\right)\in \R^{(d+1)\times(d+1)}.
          \end{equation}
We denote by $P_{s,A}:D\times(0,\infty)\to \R$ the solution to the extension problem 
          \begin{align} \label{Extendedpde1} 
          \begin{split}
          \begin{cases}
            \div( y^a B \nabla P_{s,A }) = 0, & \quad \text{in} \,\,\, D \times (0,\infty), \\
          \partial_A P_{s,A} = 0, & \quad \text{on}  \,\,\,\partial D \times [0, \infty),\\
          P_{s,A}(x,0)= p_{s,A}(x), &  \quad \text{on}  \,\,\, D.
          \end{cases}
          \end{split}
          \end{align}
In weak form \eqref{Extendedpde1} can be formulated as
\begin{equation}\label{eqn:weakform}
  \int_{\Omega}\int_{0}^\infty  \langle B \nabla P_{s,A}  ,  \nabla \phi  \rangle  y^a dy dx = c_s \int_{\Omega} L_A^s p_{s,A}  \phi(x,0) dx,   \quad  \forall \phi\in\Honeavg(D \times (0,\infty), y^adydx).  
\end{equation}

Here $c_s$ is a constant only depending on $s$, $\phi(\cdot, 0)$ is interpreted as the trace of $\phi$ on $D \times \{0 \},$ and  $\Honeavg(D \times (0,\infty), y^adydx)$ is the space of functions in $H^1(D \times (0,\infty), y^adydx)$ satisfying
 \[  \int_{D} \phi(x,y) dx =0 , \quad \text{a.e. }  y\in (0,\infty);\]

For convenience we recall that the weighted Sobolev space $H^1(D \times (0,\infty), y^adydx)$ is defined as the completion of smooth functions $\phi$ under the norm 
\begin{align*}
\lVert \phi \rVert_{H^1(D \times (0,\infty), y^adydx)}^2 & := \int_{D}\int_{0}^{\infty} \lvert \phi(x,y) \rvert^2 y^a dy dx +  \int_{D}\int_{0}^{\infty} \lvert \nabla \phi(x,y) \rvert^2 y^a dy dx
\\&=: \lVert \phi \rVert_{L^2(D \times (0,\infty), y^adydx)}^2  + [ \phi]_{H^1(D \times (0,\infty), y^adydx)}^2 .
\end{align*}
      
          \section{Bayesian Formulation of Fractional Elliptic Inverse Problems} 
          
       \label{sec:formulation}
        In this section we show continuity of the forward map under two sets of regularity conditions on the diffusion coefficient $A$  and the right-hand side $f$ of the elliptic FPDE \eqref{fractionalpde}. These conditions are found in Settings \ref{setting1} and \ref{setting2} below. Continuity of the forward map, combined with Proposition \ref{wellposed}, establishes Theorem \ref{thm:maintheorem1}. In Subsection \ref{ssec:non-smooth} (Setting \ref{setting1})  we impose no regularity on the elliptic diffusion coefficient, whereas in Subsection \ref{ssec:smoothcase} (Setting \ref{setting2}) we assume that it is differentiable. In the former setting solutions to \eqref{fractionalpde} are not necessarily continuous, while in the later setting solutions to \eqref{fractionalpde} are continuous \cite{caffarelli2016fractional}.

          \subsection{Non-smooth Case: Measurements from Bounded Linear Functionals}\label{ssec:non-smooth}
     This subsection is devoted to the proof of Theorem \ref{thm:maintheorem1} in the following setting.
          \begin{setting}\label{setting1}
          The right-hand side $f$ of \eqref{fractionalpde} is in $\Lavg$. We let $E$ be the space of matrix-valued functions $A(x)$ in $L^\infty(D : \Sp)$ for which the elliptic condition \eqref{UnifElliptic} is in place. We let $X:= (0,1)\times E$ and  $Z:= \Lavg.$ We let 
          	\begin{equation*}
 	{\cal F}: X \to Z, \quad (s,A)\mapsto p_{s,A}
 	\end{equation*} 
 which maps inputs $(s,A)\in X$ into the solution $p_{s,A}\in Z$ to the FPDE \eqref{fractionalpde}. Finally, we let ${\cal O}:Z\to \R^m$ be a bounded linear functional, and set ${\cal G}= {\cal O}\circ{\cal F}.$          
          \end{setting}

%We now state the main result of this subsection.

%\begin{thm}\label{thm:wellposedfunctionals}
%Under the conditions of Setting \ref{setting1} the forward map ${\cal F}: X\to Z$ is continuous.
%Therefore, if the prior $\mu_0$ is any measure with $\mu_0(X) = 1,$ then the Bayesian inverse problem of recovering inputs  $(s,A) \in X$ of \eqref{fractionalpde} based on data 
%$$y = {\cal G}(u) + \eta, \quad \eta \sim N(0,\Gamma),$$
%is well formulated: the posterior $\mu^y$ is well defined in $X$ and it is absolutely continuous with respect to $\mu_0.$ Moreover, the Radon-Nikodym derivative is given by \eqref{radonnikodym}. 
%\end{thm}
   
We provide two different proofs. The first one is based on standard results on the stability of the spectrum of compact self-adjoint operators. The second one relies on PDE techniques proposed in \cite{caffarelli2007extension} and \cite{caffarelli2016fractional}, where the fractional equation is interpreted as a Dirichlet to Neumman map of an appropriate elliptic equation on an extended domain.   

 \subsubsection{The Spectral Approach}\label{ssec:spectralapproach}
Let us start with the proof based on spectral methods.  For a given $A \in E$, we recall that as a map between $\Lavg$ into itself, $L_{A}^{-1}$ is a self-adjoint (with respect to the usual $L^2$-inner product) and compact operator. Furthermore, its eigenfunctions coincide with those of $L_A$ and its eigenvalues are the reciprocals of those of $L_{A}$.  Lemma \ref{lem:stability} and Proposition \ref{prop:stability} below are proved in the Appendix for completeness. They are key in the spectral proof of Theorem \ref{thm:maintheorem1}.
    
%    
%  With the above stability properties for the spectrum of $L_A^{-1}$ under small perturbations of $A$ in the $L^\infty$-norm, we may now study the stability properties of the solution to the equation \eqref{fractionalpde} under small perturbations of $A$ and $s$.  
  
%      We recall that the classical Poincar\'e inequality and a standard energy estimate yield that for $A \in E $ and $g \in L^2(D)\setminus \SpanO $, the solution $p:= L^{-1}_A g$ satisfies 
% \begin{equation}
%  \left(\int_{D}   \lvert \nabla p \rvert^2 dx\right)^{1/2} \leq \frac{C}{c_A} \lVert  g \rVert_{L^2},    
% \label{SolutionIneq}
% \end{equation}
% where $C$ is a constant that depends only on $D.$

    \begin{lem} \label{lem:stability}
    	Let $A , A'\in E$. Then, 
    	\begin{equation}
    	\lVert L_{A}^{-1} - L_{A'}^{-1} \rVert_{op} \leq \frac{C}{\lambda_A \lambda_{A'}} \lVert A - A' \rVert_{\infty},
    	\end{equation}	
    	where $\lVert \cdot \rVert_{op}$ is the operator norm for operators from $\Lavg$ into itself and $C$ is a constant that depends only on the domain $D$.  
    	\label{SpecOpNorm}
    \end{lem}
    
\begin{prop}\label{prop:stability}
    
    For every fixed $N \in \N$, $A \in E$ there exist constants $C$  and $\delta>0$ (depending on $N$ and  $A$) such that for every $A' \in E$ with $\lVert A - A'\rVert_\infty \leq \delta$ we have
	\[  \left  \lvert  \frac{1}{\lambda_i}  - \frac{1}{\lambda_i'}  \right  \rvert \leq C \lVert  A - A' \rVert_\infty, \quad i=1, \dots, N,  \]
	and
	\[   \lVert \psi_{i} - \psi_i'  \rVert_{L^2}   \leq  C   \lVert A - A' \rVert_\infty,  \quad i=1, \dots, N,\]
	\label{StabilityEigenvalues}
for some orthonormal set  $\{\psi_1, \dots, \psi_N \}$ of eigenfunctions of $L_{A}^{-1}$ with eigenvalues $\frac{1}{\lambda_1}\geq \dots \geq \frac{1}{\lambda_{N}}$ and some orthonormal set  $\{\psi_1', \dots, \psi_N' \}$ of eigenfunctions of $L_{A'}^{-1}$ with eigenvalues $\frac{1}{\lambda_1'}\geq \dots \geq \frac{1}{\lambda_{N}'}$.
\end{prop}

\begin{proof}[Proof of Theorem \ref{thm:maintheorem1}, Setting \ref{setting1}, spectral approach]

Fix $(s,A)\in X$ and $\veps>0$. Because $f\in L^2(D)$, we may pick $N \in \N$ in such a way that regardless of the orthonormal basis of eigenfunctions $\{ \psi_1, \psi_2 , \dots\}$ of $L_{A}^{-1}$ (with corresponding eigenvalues $\frac{1}{\lambda_1}, \frac{1}{\lambda_2}, \dots$) we have
\[ \max \Bigl\{  \frac{1}{\lambda_1^2},1 \Bigr\}  \sum_{i=N+1}^\infty   \langle \psi_i, f \rangle_{L^2} ^2  < \veps^2. \] 

Let us now take $A' \in E$ with $\lVert A'- A  \rVert_\infty < \delta$, where $\delta$ is as in Proposition \ref{prop:stability}, and consider two bases $\{ \psi_1, \dots, \psi_N, \dots \}$, $\{ \psi_1' , \dots, \psi_N' , \dots  \}$ of $\Lavg$ consisting of eigenfunctions of $L_{A}^{-1}$ and $L_{A'}^{-1}$ for which the first $N$ corresponding eigenfunctions are related as in Proposition \ref{prop:stability}.

Recall that $p_{s,A}:= L_{A}^{-s} f$ may be written as
\[ p_{s,A} = \sum_{i=1}^\infty   \frac{1}{\lambda_i^s} \langle  f, \psi_i \rangle_{L^2} \psi_i .\]
Now, for $s' \in (0,1)$ with $s'\geq s$, it is straightforward to check that $p_{s',A}:= L^{-s'}_{A} f$ solves the equation \eqref{fractionalpde} with fractional power $s$ but with right hand side equal to 
\[ f' := \sum_{i=1}^\infty  \frac{1}{\lambda_{i}^{s'- s}}  \langle  f , \psi_i \rangle_{L^2} \psi_i.\]
Notice that $f'$ belongs to $\Lavg$ since $s' \geq s $ and  $f \in \Lavg$. In particular, it follows that
\begin{equation}
    L_{A}^s (  p_{s,A}- p_{s',A} )  =    f' - f.
\end{equation}
Hence, 
\[      \lambda_1^{2s} \lVert  p_{s,A} - p_{s',A}  \rVert_{L^2}^2  \leq    \sum_{i=1}^\infty  \lambda_i^{2s}   \langle  p_{s,A} - p_{s',A} , \psi_i \rangle_{L^2} ^2    =   \lVert    f' - f  \rVert_{L^2}^2 ,       \]
and so
\begin{equation*}
 \lVert  p_{s,A} - p_{s',A}  \rVert_{L^2} \leq \frac{1}{\lambda_1^s} \lVert    f' - f  \rVert_{L^2} \leq \max \left\{ \frac{1}{\lambda_1} ,1\right\}\lVert    f' - f  \rVert_{L^2} .
\end{equation*}
The norm  $\lVert    f' - f  \rVert_{L^2}$ can be estimated by
\begin{align*}
\begin{split}
 \lVert    f' - f  \rVert_{L^2}^2   &   \leq     \max  \left \{ \left ( \frac{1}{\lambda_1^{s'-s}}  -1  \right)^2  ,  \left ( \frac{1}{\lambda_N^{s'-s}}  -1  \right)^2   \right\} \lVert  f  \rVert^2_{L^2}    
 \\ &  +  \max \left\{  \frac{1}{\lambda_1^{2(s'-s)}}  , 1  \right\}  \sum_{i=N+1}^\infty \langle  f, \psi_i\rangle^2_{L^2}    
 \end{split}
 \\& \leq  C_{A,N} \lvert s' - s \rvert^2    +  C_A \sum_{i=N+1}^\infty  \langle f , \psi_i \rangle^2_{L^2},
 \end{align*}
so that in particular, 
\begin{equation}
\lVert  p_{s,A} - p_{s',A}  \rVert_{L^2} \leq C_{A,N}  \lvert s' - s \rvert   + C_{A}  \left(\sum_{i=N+1}^\infty  \langle f , \psi_i \rangle^2_{L^2} \right)^{1/2} \leq  C_{A,N}  \lvert s' - s \rvert   + C_{A}  \veps.
\label{AuxWellPosed1}
\end{equation}
Changing the roles of $s'$ and $s$, we can show in a similar fashion that even when $0<s' <s $ inequality \eqref{AuxWellPosed1} is still valid. 

Let us now introduce the operators $S_N$ and $S_N'$  
\[  S_N g :=  \sum_{i=1}^N \frac{1}{\lambda_i^{s'}} \langle g ,\psi_i \rangle_{L^2} \psi_i , \quad S_N' g := \sum_{i=1}^N \frac{1}{\lambda_i^{'s'}} \langle g ,\psi_i' \rangle_{L^2} \psi_i', \]
which are truncated versions of $L_{A}^{-s'}$ and $L_{A'}^{-s'}$ respectively. It follows that, 
\begin{align*}
\begin{split}
 \lVert p_{s',A} - S_Np_{s',A}   \rVert_{L^2}^2   \leq   \frac{1}{\lambda_1^{2s'}} \sum_{i=N+1}^\infty  \langle f , \psi_i \rangle_{L^2}^2  <  \veps ^2
 \end{split}
 \end{align*}
and similarly,
\begin{align*}
\begin{split} 
\lVert p_{s',A'} - S_N'p_{s',A'}  \rVert_{L^2}^2  & \leq   \frac{1}{\lambda_1^{'2s'}} \sum_{i=N+1}^\infty  \langle f , \psi_i' \rangle_{L^2}^2
\\& \leq  \frac{1}{\lambda_1^{'2s'}}  \left( \lVert f \rVert^2_{L^2} -  \sum_{i=1}^N \langle f , \psi_i' \rangle_{L^2}^2\right)
\\& \leq  \max\left\{  \frac{1}{\lambda_1^{'2}}, 1 \right\}  \left( \lVert f \rVert^2_{L^2} -  \sum_{i=1}^N \langle f , \psi_i' \rangle_{L^2}^2 \right)
\\ & \leq C_A \left( \lVert f \rVert^2_{L^2} -  \sum_{i=1}^N \langle f , \psi_i \rangle_{L^2}^2\right)  + C_{A,N} \lVert A - A' \rVert_\infty^2
\\& \leq C_A\veps^2 + C_{A,N}\lVert A - A' \rVert_\infty^2
\end{split}
\end{align*}
where the fourth inequality follows from Proposition \ref{StabilityEigenvalues}. The above inequalities combined with Proposition \ref{prop:stability} imply that 
\[  \lVert S_N p_{s',A}   - S_N' p_{s',A} \rVert_{L^2} \leq C_{A,N} \lVert A - A' \rVert_\infty.  \] 
%On the other hand, we can estimate $ \lVert  p_{s', A} - p_{s', A'}  \rVert_{L^2}  $ as follows. First, 
%\[p_{s',A} = \sum_{i=1}^N \frac{1}{\lambda_i^{s'}} \langle f ,\psi_i' \rangle_{L^2} \psi_i'    + \sum_{i=N+1}^\infty \frac{1}{\lambda_i^{s'}} \langle f ,\psi_i' \rangle_{L^2} \psi_i'     =: S_N(p_{s',A}) + ( p_{s',A} - S_N(p_{s',A})    ).  \]
%From Proposition \ref{StabilityEigenvalues} 
Combining all the previous inequalities with \eqref{AuxWellPosed1}, we conclude that
\begin{equation}
\lVert p_{s,A} - p_{s',A'}   \rVert_{L^2}  \leq  C_{ A,N} \left(  | s- s' | + \lVert A - A' \rVert_\infty  \right)  + C_{A} \veps . 
\label{AuxWellPosed2}
\end{equation}
Therefore, for all $(s',A')$ satisfying  $ | s- s' | + \lVert A - A' \rVert_\infty < \min \left\{ \frac{\veps}{ C_{A,N}}, \delta \right\} $, 
\[  \lVert \mathcal{F}(s,A) - \mathcal{F}(s',A') \rVert_{L^2}< C_A\veps,\]
and continuity of $\mathcal{F}$ is proved.
\end{proof}

\subsubsection{The Extension Approach}\label{ssec:extensionapproach}

Let us now consider the second proof of Theorem \ref{thm:maintheorem1} in Setting \ref{setting1}. This proof serves as an alternative to studying the stability of the spectra of $L_{A}^{-1}$.
\begin{proof}[Proof of Theorem \ref{thm:maintheorem1}, Setting \ref{setting1}, extension approach]

Let $P_{s,A}$ and $P_{s,A'}$ be the solutions to \eqref{Extendedpde1} with inputs $(s,A)$ and $(s,A')$, respectively. Using the test function $\phi:= P_{s,A} - P_{s,A'}$ in the associated weak formulations \eqref{eqn:weakform}   we deduce that
\begin{align*}
 \int_{\Omega}\int_{0}^\infty  \langle B \nabla P_{s,A}  ,  \nabla (P_{s,A} - P_{s,A'})   \rangle  y^a dy dx &= c_s \int_{\Omega} L_A^s p_{s,A}(p_{s,A} - p_{s,A'}) dx
\end{align*}     
and that
\[   \int_{\Omega}\int_{0}^\infty  \langle B' \nabla P_{s,A'}  ,  \nabla (P_{s,A} - P_{s,A'})  \rangle  y^a dy dx = c_s \int_{\Omega} L_{A'}^s p_{s,A'}(p_{s,A} - p_{s,A'}) dx.\]
Since both $L_{A}^s p_{s,A}$ and $L_{A'}^s p_{s,A'}$ are equal to $f$, we deduce that
\[  \int_{\Omega}\int_{0}^\infty  \langle B \nabla P_{s,A}  ,  \nabla (P_{s,A} - P_{s,A'})   \rangle  y^a dy dx  = \int_{\Omega}\int_{0}^\infty  \langle B' \nabla P_{s,A'}  , \nabla (P_{s,A} - P_{s,A'})  \rangle  y^a dy dx.   \]
Subtracting $\int_{\Omega}\int_{0}^\infty \langle   B \nabla P_{s,A'}  ,   \nabla (  P_{s,A} - P_{s,A'}  )   \rangle  y^a dy dx$          
from both sides of the above equation we obtain
\[  \int_{\Omega}\int_{0}^\infty  \langle B \nabla (P_{s,A}  - P_{s,A'}) ,  \nabla (P_{s,A} - P_{s,A'})   \rangle  y^a dy dx  = \int_{\Omega}\int_{0}^\infty  \langle (B' -B) \nabla P_{s,A'}  ,  \nabla (P_{s,A} - P_{s,A'})   \rangle  y^a dy dx   \]         
From this it follows that
\begin{align}
\begin{split}
  &\min\{1 , \lambda_A \} \int_{\Omega}\int_{0}^\infty   \lvert  \nabla ( P_{s,A} - P_{s,A'}   )  \rvert^2 y^a dy dx  
  \\ & \leq  \lVert A - A' \rVert_{\infty}  \left( \int_{\Omega}  \int_{0}^\infty \lvert \nabla P_{s,A'} \rvert^2 y^a dy dx  \right)^{1/2} \left( \int_{\Omega}\int_{0}^\infty   \lvert  \nabla ( P_{s,A} - P_{s,A'}   )  \rvert^2 y^a dy dx \right)^{1/2}.  
  \end{split}
\end{align}
Therefore, for $A'$ such that $\|A- A'\|_\infty \le \lambda_A/2,$
\begin{align}
\begin{split}
\left( \int_{\Omega}\int_{0}^\infty   \lvert  \nabla ( P_{s,A} - P_{s,A'}   )  \rvert^2 y^a dy dx \right)^{1/2} & \leq \max \left\{1, \frac{1}{\lambda_A} \right\}   \lVert A - A' \rVert_{\infty}  \left( \int_{\Omega}  \int_{0}^\infty \lvert \nabla P_{s,A'} \rvert^2 y^a dy dx  \right)^{1/2}  
\\ & = c_s \max \left\{1, \frac{1}{\lambda_A} \right\}   \lVert A - A' \rVert_{\infty}  \left( \int_{\Omega}  L_{A'}^s p_{s,A'}  p_{s,A'}dx  \right)^{1/2}
\\& \leq c_s \max \left\{1, \frac{1}{\lambda_A} \right\}   \lVert A - A' \rVert_{\infty}   \frac{1}{\lambda_1'^{s/2}}  \lVert  f \rVert_{L^2}     
\\& \leq c_s \max \left\{1, \frac{1}{\lambda_A} \right\}   \lVert A - A' \rVert_{\infty}   \frac{1}{\lambda_A^{s/2}}  \lVert  f \rVert_{L^2} 
\end{split}
\label{AuxExtension1}
\end{align}
where the first equality follows using the  weak formulation for the equation satisfied by $P_{s,A'}$ taking as test function $\phi:= P_{s,A'}$ and the second to last inequality follows from the fact that $L_{A'}^s p_{s,A'} = f$, and the last inequality follows from the variational representation of the first eigenvalue of $L_{A'}^{-1}$ and equation \eqref{ineqmineig}.

It has been shown that the trace of the weighted Sobolev space $H^1( D \times (0,\infty), y^ady dx)$ is the space $H^s$, see Section 7 of \cite{caffarelli2016fractional} and references within. In particular, there exists a constant $C_s$ depending only on $s$ such that
\[ [ p_{s,A}- p_{s,A'}]_{H^s} \leq  \lVert p_{s,A}- p_{s,A'}\rVert_{H^s}  \leq C_s  \lVert P_{s,A} - P_{s,A'} \rVert_{H^1(\Omega\times (0,\infty), y^adydx)}.     \]
On the other hand, there exists a constant $C_D$ (only depending on the domain $D$) such that
\begin{equation}
 \lVert P_{s,A} - P_{s,A'} \rVert_{H^1(\Omega\times (0,\infty), y^adydx)} \leq  C_D  [  P_{s,A} - P_{s,A'} ]_{H^1(\Omega\times (0,\infty), y^adydx)}.
 \label{PoincareH1ya}
 \end{equation}
Indeed, this follows from the following considerations. First, by Fubini's theorem the function $P_{s,A}(\cdot, y) -P_{s,A'}(\cdot , y) $ belongs to $H^1(D)$ for almost every $y \in (0,\infty)$. Second, for a.e. $y \in(0,\infty)$ both $P_{s,A}(\cdot, y)$ and $P_{s,A'}(\cdot, y)$ have average (in $x$) equal to zero, and hence by Poincar\'e inequality in $D$
\[ \int_{D} \lvert P_{s,A}(x,y ) - P_{s,A'}(x,y) \rvert^2 dx \leq C_D \int_{D} \lvert \nabla_{x}( P_{s,A}(x,y) - P_{s,A'}(x,y))  \rvert^2 dx,  \text{ a.e. } y \in (0,\infty). \]
Integration with respect to $y^a dy$ yields
\begin{align*}
\begin{split}
\int_{0}^\infty \int_{D} \lvert P_{s,A} -P_{s,A'} \rvert^2 y^a dxdy & \leq C_D \int_{0}^\infty \int_{D} \lvert \nabla_x (P_{s,A}(x,y ) - P_{s,A'}(x,y)) \rvert^2 y^a dxdy 
\\& \leq  C_D \int_{0}^\infty \int_{D} \lvert \nabla (P_{s,A}(x,y ) - P_{s,A'}(x,y))\rvert^2y^adxdy,
\end{split}
\end{align*} 
thus establishing \eqref{PoincareH1ya}. From \eqref{AuxExtension1} and \eqref{PoincareH1ya}, we deduce that 
\[ [p_{s,A} - p_{s,A'}]_{H^s} \leq C_{D,s} \max \left\{1, \frac{1}{\lambda_A} \right\}   \lVert A - A' \rVert_{\infty}   \frac{1}{\lambda_1^s}  \lVert  f \rVert_{L^2}. \]
Finally we use the fact that both $p_{s,A}$ and $p_{s,A'}$ have average zero and Poincar\'e inequality in $H^s$ (which follows for instance from a compactness argument and Theorem 7.1 in \cite{di2012hitchhikers}) to conclude that 
\begin{equation}\label{eqn:4.8}
\lVert p_{s,A} - p_{s,A'}\rVert_{H^s} \leq  C_s[p_{s,A} - p_{s,A'}]_{H^s} \leq    C_{D,s} \max \left\{1, \frac{1}{\lambda_A} \right\}   \lVert A - A' \rVert_{\infty}   \frac{1}{\lambda_1^s}  \lVert  f \rVert_{L^2}.
\end{equation}

The above shows that for every fixed $s$ the map $ A \mapsto \mathcal{F}(s,A)$ is locally Lipschitz continuous when the range is not only endowed with the $L^2$-norm, but in fact when it is endowed with the stronger $H^s$-norm. The continuity of the map $\mathcal{F}$ in the first coordinate (i.e. fixing $A$ and changing $s$) can be obtained directly from the representation of $L_A^s$ in terms of eigenvalues and eigenfunctions.
\end{proof}

\begin{remark}
If the marginal of the prior $\mu_0$ on its first coordinate was supported in $[\bar{s} , 1]$ for $\bar{s}>0$, then we could take the space $Z$ to be equal to $H^{\bar{s}}$ with norm $\lVert \cdot \rVert_{H^{\bar{s}}}$. The above proof shows that the forward map $\mathcal{F}$ is indeed continuous. Notice that in that case, the observation map $\mathcal{O}$ may be constructed using $H^{-\bar{s}}$; the bottom line is that the map $\mathcal{G}= \mathcal{O}\circ \mathcal{F}$ is continuous. 
\end{remark}

\begin{remark}\label{remark:boundnorms}
As in \eqref{AuxExtension1} and \eqref{eqn:4.8}, and using the trace theorem for $H^1(D \times (0,1), y^a dy dx)$ we deduce that
\[ [ p_{s,A}] _{H^s}  \leq \lVert  p_{s,A} \rVert_{H^s} \leq C_s \lVert P_{s,A} \rVert_{H^1} =  C_s [ P_{s,A} ]_{H^1} + C_s \lVert  P_{s,A} \rVert_{L^2} \leq      C_s  \| p_{s,A} \|_{\mathcal{H}_A^s}.   \]
Since we also have 
\[    [  p_{s,A} ]^2_{\mathcal{H}^s_A}   = \sum_{i=1}^\infty  \frac{\langle  f ,  \psi_i\rangle^2_{L^2}}{\lambda_i^s}  \leq \frac{1}{\lambda_1^s}  \lVert f \rVert_{L^2}^2,  \]
we conclude that
\begin{equation}\label{eq:remarkhs}
 [ p_{s,A} ]_{H^s} \leq   \frac{C_s}{\lambda_1^{s/2}}  \lVert f \rVert_{L^2}\le \frac{C_s}{\lambda_A^{s/2}} \lVert f \rVert_{L^2}.
\end{equation}
We will use this identity in the next subsection when we discuss the continuity of the map $\mathcal{G}$ in the pointwise observations case.
\end{remark}

 \subsection{Smooth Case: Pointwise Observations}\label{ssec:smoothcase}
          
      The presentation of this subsection is parallel to that of the previous one. Here we prove Theorem \ref{thm:maintheorem1} under the following setting:
          \begin{setting}\label{setting2}
          We let $f\in \Lpavg$ for some $p \geq 2,$ and let  $s_p := \frac{d}{2p}.$ We let $X:= (s_p,1] \cap C^1( \overline{D} : \mathbb{S}_{+}(\R^d)),$ and let $Z:= C^0(D).$ 
         As before, we let 
          	\begin{equation*}
 	{\cal F}: X \to Z, \quad (s,A)\mapsto p_{s,A}
 	\end{equation*} 
  map inputs $(s,A)\in X$ to the solution $p_{s,A}\in Z$ to the fractional PDE \eqref{fractionalpde}. Finally, we let ${\cal O}:Z\to \R^m$ be a bounded linear functional, and set ${\cal G}= {\cal O}\circ{\cal F}.$          
          \end{setting}        
          
          \begin{remark}	The assumptions on $f$ and $A$ in Setting \ref{setting2} guarantee that the solution $p_{s,A}$ is continuous \cite{caffarelli2016fractional}.  Pointwise observations of the solution constitute a natural example of observation map  $\mathcal{O}.$ That is, 
 	\[  \mathcal{O}(p) := \left( \begin{matrix}   p(x_1) \\ \vdots \\ p(x_m)    \end{matrix} \right), \]         
 	for some fixed points $x_1, \dots, x_m \in D$.
 	
          To ease the notation we will often write $C^1$ instead of $C^1( \overline{D} : \mathbb{S}_{+}(\R^d) ).$  No confusion will arise from doing so. Note that $C^1$ is contained in the space $E$ of uniformly elliptic matrix-valued functions of Setting \ref{setting1}.       
          \end{remark}
          
          We are ready to prove Theorem \ref{thm:maintheorem1} in the above setting. The proof combines the analysis of the non-smooth case in Subsection \ref{ssec:non-smooth} with the regularity theory of Caffarelli and Stinga \cite{caffarelli2016fractional} to show continuity of the forward map. The formulation of the inverse problem then follows from Proposition \ref{wellposed}.

  \begin{proof}[Proof of Theorem \ref{thm:maintheorem1}, Setting \ref{setting2}]  
   	Let $A \in C^1$, $s \in (s_p,1]$ and let $\veps>0$. Consider $p_{s,A}:= L_A^{-s} f$. We have already shown that there exists $\delta>0$ such that if $\lvert s-s' \rvert  + \lVert A - A' \rVert_\infty < \delta$  then 
\[ \lVert p_{s,A} - p_{s',A'} \rVert_{L^2} \leq  \veps. \] 
Now, from \cite{caffarelli2016fractional} it follows that $p_{s,A} = L^{-s}_A f $  is an $\alpha$-H\"{o}lder continuous function with $\alpha =2s - \frac{d}{p}$ if $2s - \frac{d}{p} <1,$ and $\alpha$-H\"{o}lder continuous for any $\alpha<1$ if $2s - \frac{d}{p}\ge 1$. Moreover, its H\"{o}lder seminorm is bounded by
\begin{equation}
 [p_{s,A}]_{C^{0,\alpha}(D)}  \leq   C( \lVert p_{s,A} \rVert_{L^2}  + [p_{s,A}]_{H^s} + \lVert  f \rVert_{L^p}  ),
\label{RegEstimates}
\end{equation}
for a constant $C$ depending on $s$, the modulus of continuity of $A$ and the ellipticity of $A$  (see  Theorem 1.2 in \cite{caffarelli2016fractional}). Using Remark \ref{remark:boundnorms} we deduce that
\[ [p_{s,A}]_{C^{0,\alpha}(D)} \leq C  \lVert f  \rVert_{L^p},  \]
for a constant that depends on $s$, on the modulus of continuity of $A$, the ellipticity of $A$, and the first eigenvalue of $L_A$.

Hence, for any $(s',A') \in X$  
\[  [p_{s',A'}]_{C^{0,\alpha}(D)}  \leq   C( \lVert p_{s',A'} \rVert_{L^2}  + [p_{s',A'}]_{H^s} + \lVert  f \rVert_{L^p}  ),  \]
where the constant $C$ may be taken to be uniform over all $s',A'$ with  $ \lvert s- s' \rvert +\lVert A - A' \rVert_{C^1} \leq \delta$ for some $\delta>0$ small enough.

On the other hand, a simple argument shows that for arbitrary $\alpha$-H\"{o}lder continuous functions $\phi_1, \phi_2$ we have
\[ \lVert \phi_1 - \phi_2 \rVert_\infty  \leq C_{D,\alpha} \max \left\{ [\phi_1]_{C^{0,\alpha}} , [ \phi_2]_{C^{0,\alpha}}   \right\}^{d/(2\alpha + d)} \cdot \lVert \phi_1 - \phi_2 \rVert_{L^2}^{2\alpha / (2 \alpha + d)}  ,\]          
 where $C_{D,\alpha}$ is a constant that depends on $\alpha$ and the domain $D$. We can then deduce that  if $\lvert s - s' \rvert + \lVert A - A' \rVert_{C^1} <  \delta$  then, 
\[ \lVert  p_{s,A} - p_{s',A'} \rVert_\infty   \leq  C \lVert f \rVert_{L^p}^{d/ (2 \alpha + d)} \veps^{2 \alpha/ (2 \alpha + d)}. \]

This shows the continuity of the forward map $\mathcal{F}: X \rightarrow Z$ when $X=( s_p ,1)  \times C^1  $ and $Z = C^0(D)$ is endowed with the sup norm.    
\end{proof}

\section{Hellinger Continuous Dependence on Observations}\label{sec:hellinger}

We now study the Hellinger continuity of the posterior distribution $\mu^y$ as $y$ changes. In light of Proposition \ref{propositionhell} and the discussion proceeding it, we impose some conditions on the priors $\mu_0$ for which we can guarantee the stability of posterior distributions.

\begin{assumption}\label{assumptionprior}
	In what follows let  $s_{-}=0$ for Setting \ref{setting1}, and $s_{-}=s_p$ for Setting \ref{setting2}. We assume that the prior $\mu_0$ satisfies $\mu_0(X) = 1$ and that its marginals $\mu_{0,1}$ and $\mu_{0,2}$ on the first and second variables satisfy
\begin{enumerate}	
\item For every polynomial $q$ in two variables   
\[  \int\left|q \left(\frac{1}{\lambda_A}, \lVert A\rVert   \right) \right| d \mu_{0,2}(A)   < \infty. \]
\item  The support of $\mu_{0,1}$ is contained in $[s_{-}+\veps, 1 - \veps]$ for some small enough $\veps>0$ .
\end{enumerate}
\end{assumption}

Notice that with these assumptions, $\mu_{0,2}$ is allowed to give positive mass to sets of $A$ with large ellipticity and large norm, provided the mass decays fast enough. In particular, $\mu_{0,2}$ can be chosen to be log-Gaussian (like in Remark \ref{RemarkLogGaussian} below). In contrast, notice that the assumptions on $\mu_{0,1}$, that we make for simplicity, rule out any possible degeneracy in the estimates obtained in Section \ref{sec:formulation} as $s$ approaches the endpoints of the interval $[s_{-},1]$.

%Let us consider the map 
%\[   T :  X \rightarrow   (0,\infty)      \]
%\[  T : (s,A) \mapsto  c_A:= \inf_{x \in D}  \lambda_{A,1}(x),       \]
% where in the above $\lambda_{A,1}(x)$ is the smallest eigenvalue of the matrix $A(x)$.  It can be easily checked that the map $T$ is continuous. In particular, the map $T$ is measurable and we may consider the \textit{push-forward} measure $T_{\sharp} \mu_0 $, which can be interpreted as the distribution of the ellipticity constant for $A$. 
%% 
% \begin{assumption}\label{assumptionprior}
%We assume that the prior $\mu_0$ satisfies $\mu_0(X) = 1$ and 
%\[  \int_{X} \frac{1}{c_A} d \mu_0(s,A)  =   \int_{0}^\infty \frac{1}{\lambda } d(T_{\sharp }\mu_0  )(\lambda) < \infty. \]
%\label{AssumpPrior}
% \end{assumption}

 \begin{remark}
 	\label{RemarkLogGaussian}
 One example of interest of a prior $\mu_0$ for which $\mu_{0,2}$ satisfies the first condition in Assumption \ref{assumptionprior} is the following. Suppose that we pick $A$ according to
 \[  A (x) : = e^{-v(x)}  I_d,  \] 
 where we recall $I_d$ is the identity matrix and where $v$ is chosen according to a Gaussian process in the appropriate Banach space. It follows that 
 \[  \frac{1}{\lambda_A}  \leq  e^{ \lVert  v \rVert} \quad \text{and } \quad \lVert A \rVert \leq  e^{ \lVert  v \rVert}. \]
We conclude that for any polynomial $q$ in two variables, 
\[  \int_E\left \lvert q \left(\frac{1}{\lambda_A}, \lVert A\rVert   \right)  \right\rvert d \mu_{0,2}(A)   <\infty,  \]
thanks to properties of Gaussian processes.
 \end{remark}

\begin{proof}[Proof of Theorem \ref{thm:maintheorem2}]
Since in both Settings \ref{setting1} and \ref{setting2} we assume that the observation map ${\cal O}$  is a bounded functional, the proof reduces to showing that ${\cal F}\in L^2_{\mu_0}(X)$ by Proposition \ref{propositionhell}.

The proof in Setting \ref{setting1} is straightforward, since for every $(s,A) \in X$ we have
\begin{equation*}
\|{\cal F}\bigl(s,A \bigr) \|_{L^2} \le \frac{\|f\|_{L^2}}{\lambda_{1}^s}  \le \max \left\{1, \frac{1}{\lambda_{1}} \right\}  \|f\|_{L^2} \leq \max \left\{1, \frac{C_D}{\lambda_A} \right\}  \|f\|_{L^2} ,
\end{equation*}
where we are using $\lambda_{1}$ to denote the first (nonzero) eigenvalue of $L_{A}$; the last inequality follows from the variational formula for the first nonzero eigenvalue of $L_{A}$  (see the Appendix). Assumption \ref{assumptionprior} on $\mu_0$ guarantees  that $\cal {F}$ belongs to $L^2_{\mu_0}(X)$. Notice that in this case we are not fully using Assumption \ref{assumptionprior}. Indeed, we are only using the fact that $1/ \lambda_A$ is square integrable.

Let us now consider the proof in setting \ref{setting2}. We base our analysis on equations \eqref{eq:remarkhs} and \eqref{RegEstimates}. Recall that the constant appearing in the regularity estimates in \eqref{RegEstimates} depends on the ellipticity of $A$, the modulus of continuity of $A$ and $s$, but their explicit dependence is not provided in \cite{caffarelli2016fractional}. For our purposes, understanding the dependence of the constants on the parameters of the problem is important. Fortunately, the dependence of the constants in the estimates in \cite{caffarelli2016fractional} can be tracked down by modifying slightly some of the crucial estimates and inequalities in \cite{caffarelli2016fractional}. We point out the steps that need to be taken to achieve this.

First, the constant in the Caccioppoli inequality in Lemma 3.2 in \cite{caffarelli2016fractional} can be written as
\[ C = \tilde{C} \left(\frac{\Lambda}{\lambda }\right)^2,  \] 
where $\tilde{C}$ is universal and in particular does not depend on $A$. This follows directly from the proof provided in \cite{caffarelli2016fractional}, given that the ``Cauchy inequality with $\veps>0$'' can be applied with $\veps =\frac{\lambda}{4(2 \Lambda +1 ) } $.

Second, the approximation Lemma (Corollary 3.3 in \cite{caffarelli2016fractional}) is deduced using the Caccioppoli estimate mentioned above and a compactness argument. It is at this stage that one loses the explicit dependence of constants on ellipticity. One can actually restate this approximation lemma by modifying the definition of ``normalized solution'' $U$ given in the mentioned corollary. Indeed, one can replace the normalization condition to 
\[ \int_{B_1}U(x,0)^2 dx + \int_{B_1^*} U^2 y^a dy dx \leq \left(\frac{\lambda}{\Lambda}\right)^2. \]
Then, the same argument in \cite{caffarelli2016fractional}, shows that the number $\delta$ can be chosen to have the form
\[ \delta= \tilde{\delta}(\veps) \cdot \frac{\lambda}{\Lambda}.\]

With these modifications, one can follow the arguments in \cite{caffarelli2016fractional} to ultimately prove that the constant in \eqref{RegEstimates} depends polynomially on the ellipticity and the modulus of continuity of $A$. This dependence on ellipticity and modulus of continuity is all we need to guarantee that for $\mu_0$ satisfying Assumption \ref{assumptionprior}, the forward map $\mathcal{F}$ belongs to $L^2_{\mu_0}(X)$. 

\end{proof}

          \section{Example} \label{sec:example}
The purpose of this section is to illustrate some aspects of the theory in an analytically tractable set-up. A more thorough numerical investigation is beyond the scope of this work.  Unlike in the rest of the paper, we assume here $A\equiv 1$ to be known, so that $\Las = -\Delta^s$. Thus we restrict our attention to the case where the only uncertainty arises from the order $s,$  and we will henceforth  drop $A$ from the notation. 
\subsection{Framework}
We consider the toy model
\begin{align}\label{toyexample}
\begin{cases}
-\Delta^s p &= \cos(bx)-\frac{2}{b}\sin(b\pi),  \quad \text{in} \,\,\, [-\pi,\pi], \\ 
p'(-\pi)&=p'(\pi)=0.
\end{cases}
\end{align}
This framework allows us to easily write the eigenfunctions and eigenvalues of the Neumann Laplacian:
\begin{equation}
\phi_k(x) = \cos(kx), \quad \quad \lambda_k = k^2, \quad k=1, 2, \ldots.
\end{equation}
Moreover $f(x) :=  \cos(bx)-\frac{2}{b}\sin(b\pi)$ has zero mean and admits the Fourier representation
$$f(x) = \sum _{k=1}^\infty \frac{(-1)^k}{b^2-k^2} \cos(kx), \quad -\pi\le x \le \pi.$$
The solution $p= p_s$ to \eqref{toyexample} is thus given by 
\begin{equation*}
p_s(x) = \sum_{k=1}^\infty  \frac{(-1)^k}{k^{2s}(b^2 - k^2)}\cos(kx), \quad -\pi\le x \le \pi.
\end{equation*}
This expression gives a tangible representation of the forward map $\F(s) := p_s.$ Note that, for any $0<s\le 1,$ $p_s$ is continuous and thus can be evaluated pointwise. We now describe our observation model.

Let $m\ge 1.$ We define a grid of $m$ points in $[-\pi,\pi]$ as follows. If $m=1$ we let $x_1=\pi.$ If $m\ge 2$ we let $h = 2\pi/(m-1)$ and set $x_j = -\pi+ jh, \, j =0,\ldots, m-1.$ We then define, for continuous $p$, ${\cal O}(p):= [p(x_1), \ldots, p(x_m)]^T,$ and ${\cal G} = {\cal O}\circ {\cal F}.$

 With the above definitions the inverse problem of interest can be written as 
\begin{equation}\label{invtoy}
y = {\cal G}(s) + \eta, 
\end{equation}
where  ${\cal G}(s) = [p_s(x_1), \ldots, p_s(x_m)]^T\in \R^m,$ and $\eta$ is assumed to be Gaussian distributed, $\eta\sim N(0,\gamma^2 I_{m\times m}).$

\subsection{Numerical Illustrations}
In this subsection we represent some posterior measures on the order $s$ of equation \eqref{toyexample}. We put a uniform prior on $s$,  $s\sim \mu_0\equiv {\cal U}[0,1],$ and generate synthetic data from the underlying parameter $s^* = 0.7,$ according to the model \eqref{invtoy} with $b=1/2.$ We consider three different values for: i) the number of pointwise observations $m$; and (ii) the level of the noise in the observations, $\gamma.$  The results are shown in Figure  \ref{figurepost}. The posterior is expected to concentrate around the ``true" underlying parameter $s=0.7$ in the regimes $m\to\infty$ or $\gamma\to 0,$ and this can be observed in Figure \ref{figurepost}. We remark that the plotted posteriors $\mu^y(ds)$ depend of course of the realization of the data $y$, and the plots below represent only one such realization. We conducted many experiments all of which demonstrate this concentration phenomena. In particular, it is perhaps surprising that even with one pointwise observation $(m=1)$ accurate reconstruction of the order $s$ is possible provided that the observation noise $\gamma$ is small enough.
\begin{center}
\begin{figure}[h!] 
\centering
 \subfloat{%
  \begin{tabular}{c}
   \includegraphics[width=.28\linewidth]{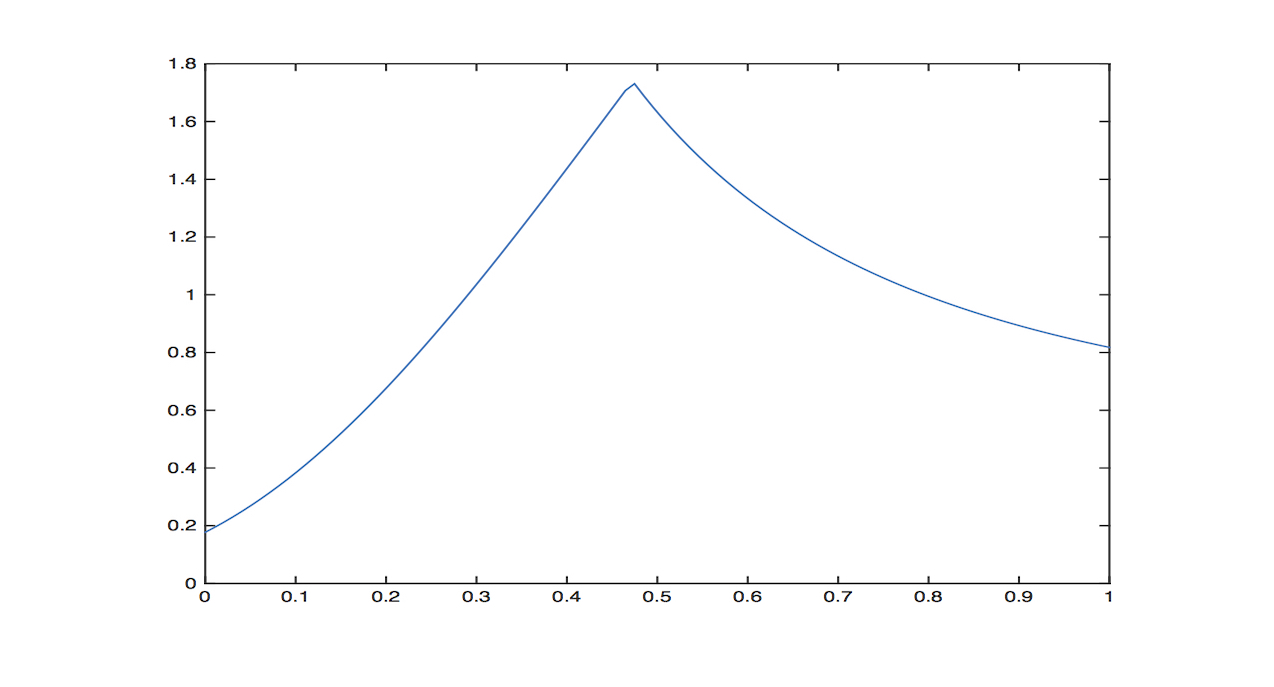} \\
   \includegraphics[width=.28\linewidth]{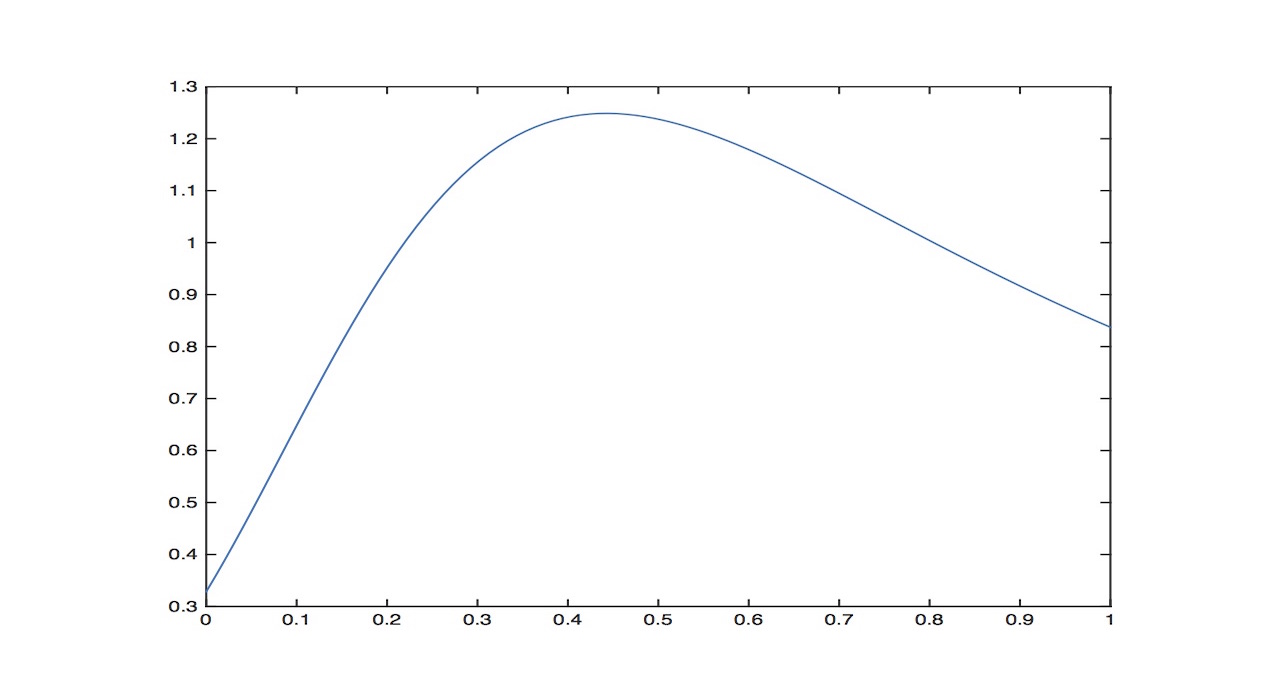} \\
   \includegraphics[width=.28\linewidth]{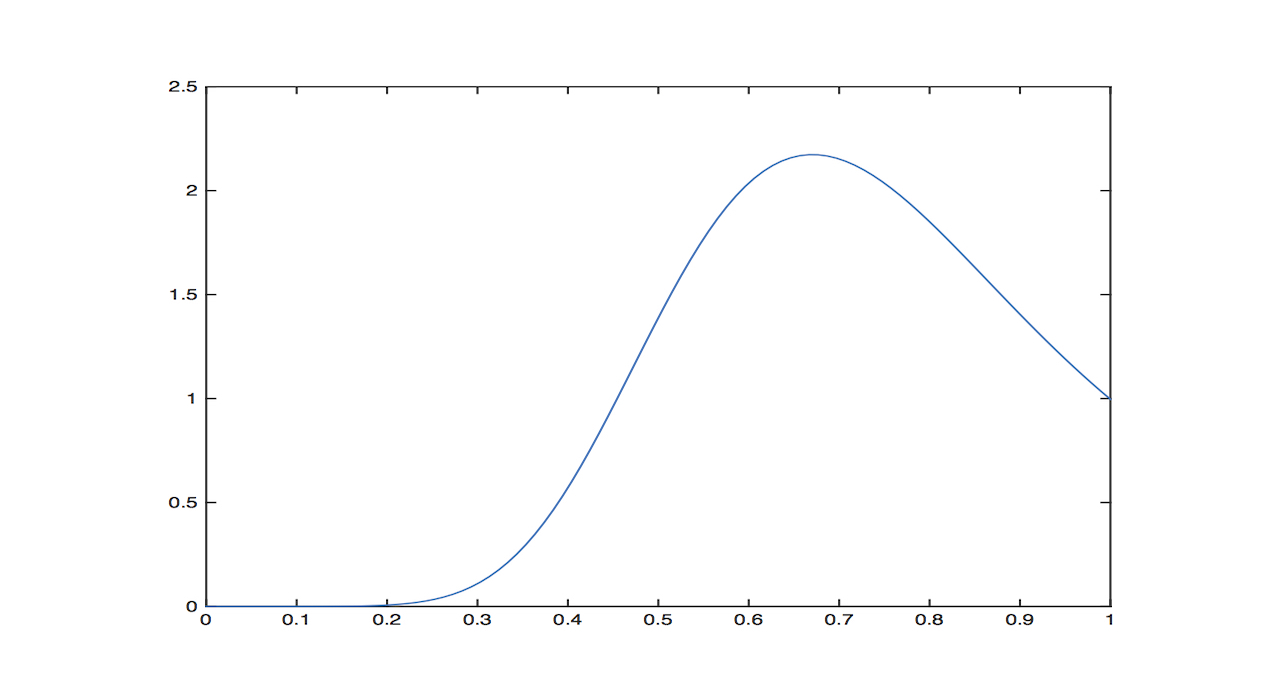}
  \end{tabular}
 }%
 \subfloat{%
  \begin{tabular}{c}
   \includegraphics[width=.28\linewidth]{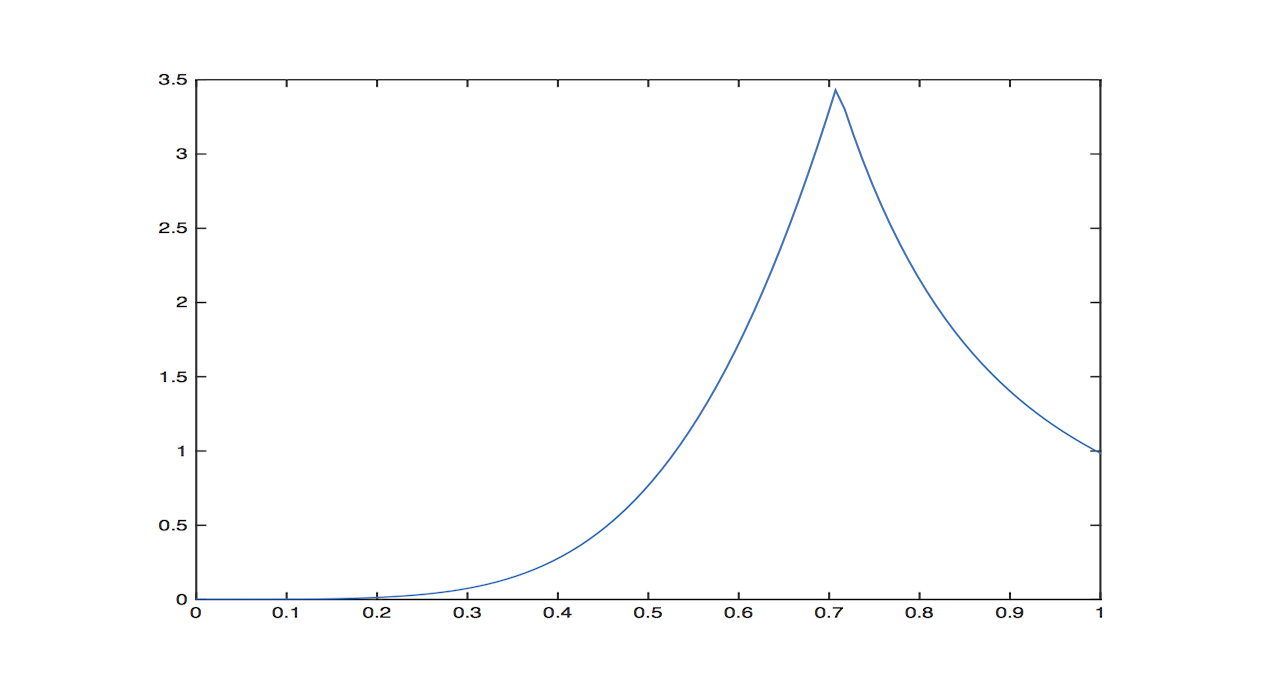} \\
   \includegraphics[width=.28\linewidth]{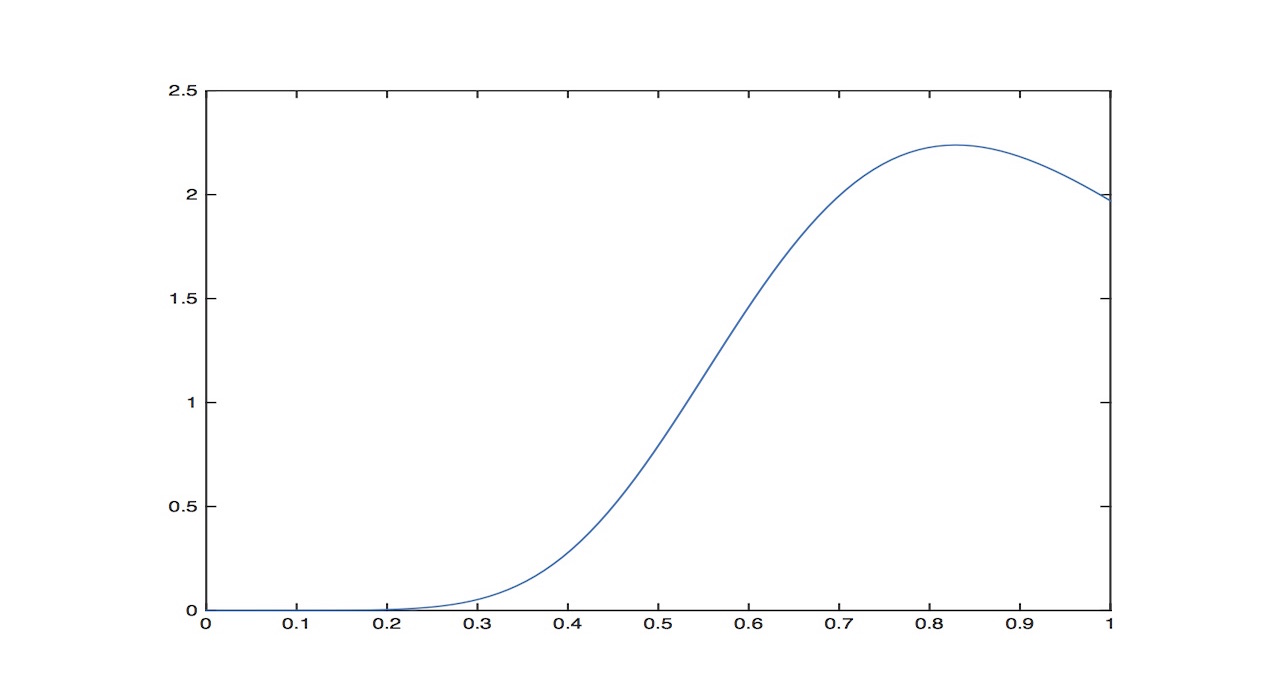} \\
   \includegraphics[width=.28\linewidth]{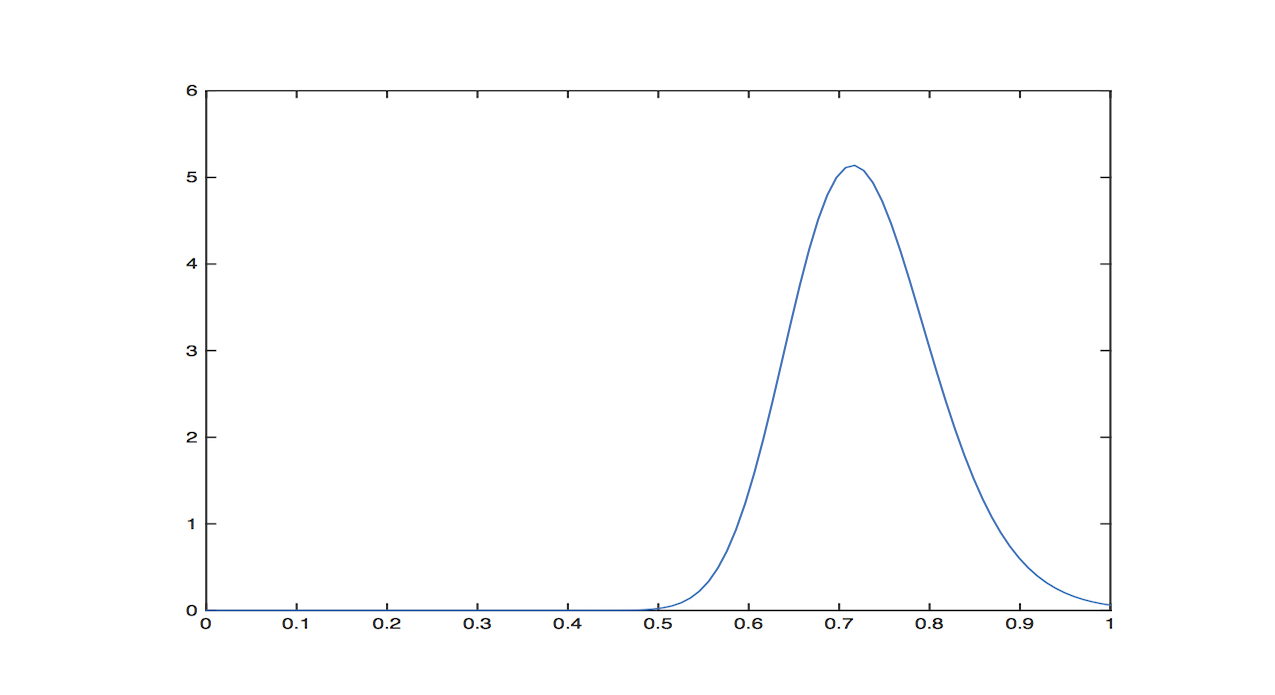}
  \end{tabular}
 }%
 \subfloat{%
  \begin{tabular}{c}
   \includegraphics[width=.28\linewidth]{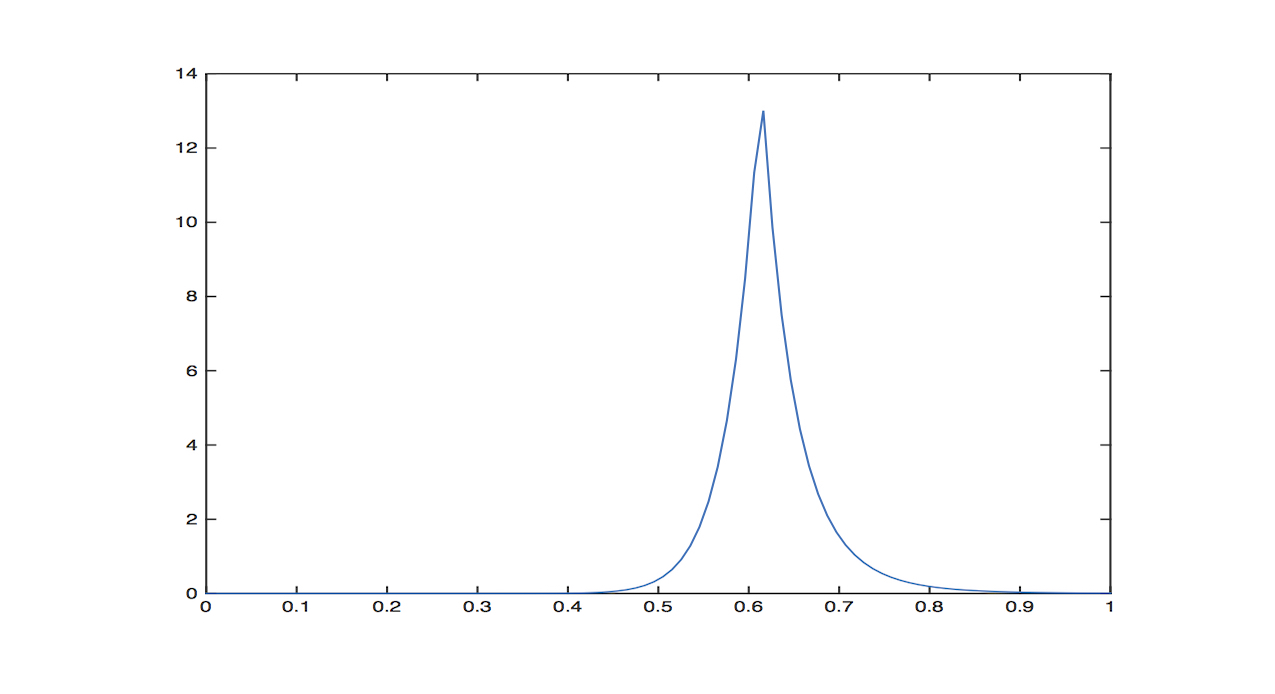} \\
   \includegraphics[width=.28\linewidth]{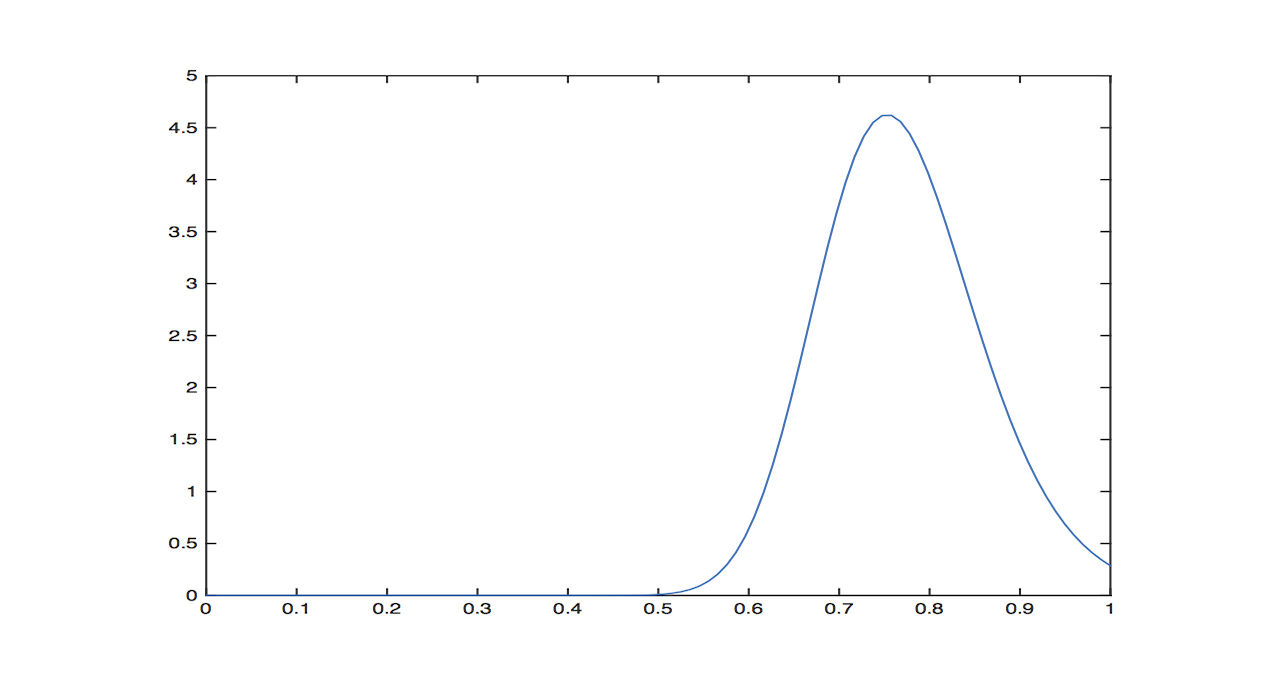} \\
   \includegraphics[width=.28\linewidth]{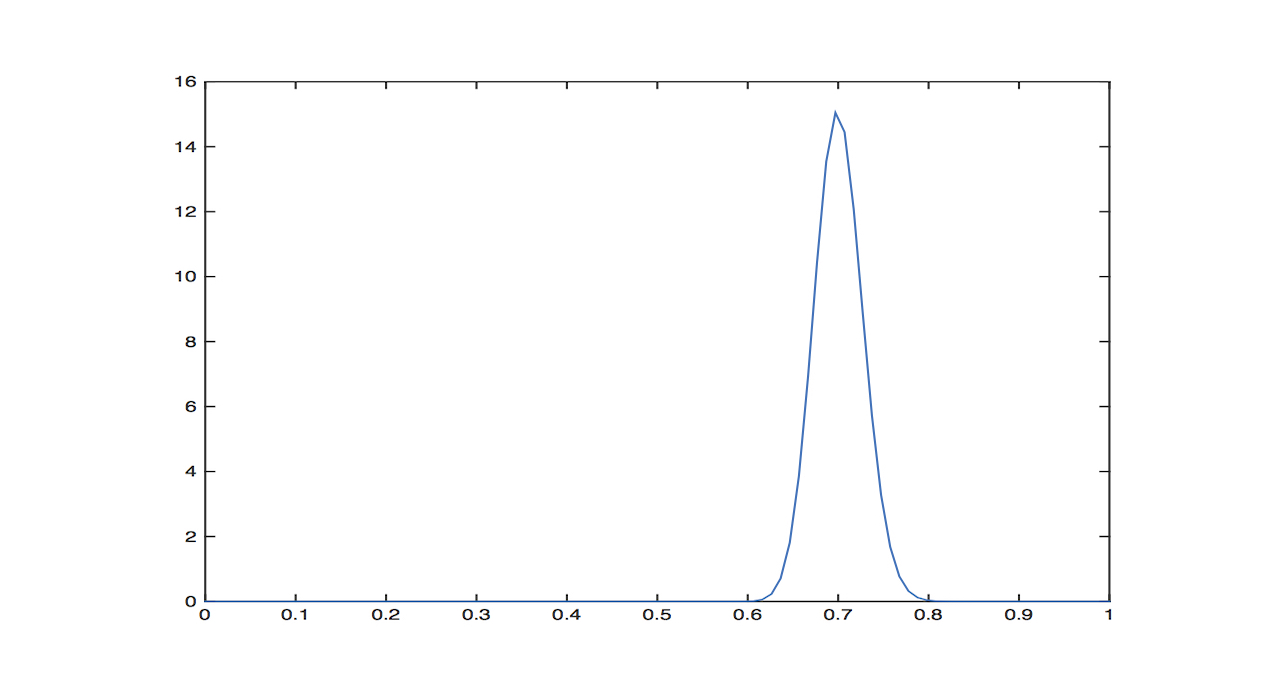}
  \end{tabular}
 }%
\caption{\label{figurepost}Posterior measures on the order $s.$ The data is generated from $s^*=0.7.$ The three rows correspond, respectively, to $m=1,\, 100, \, 10000.$ The three columns correspond to $\gamma=0.3, \, 0.15, \, 0.075.$}
\end{figure}

\end{center}

\section{Conclusions} \label{sec:conclusion}
We conclude by summarizing the main outcomes of this work, and by describing some future research directions.
\begin{itemize}
\item We provide the fundamental framework for Bayesian learning of the order and diffusion coefficient of a  FPDE. 
\item We combine two thriving research areas: the Bayesian formulation of inverse problems in function space, and the analysis of regularity theory for elliptic FPDEs.
\item We generalize the Bayesian formulation of (integer order) elliptic inverse problems \cite{dashti2011uncertainty} to the fractional case. We also generalize the existing theory by considering a matrix-valued ---rather than scalar--- diffusion coefficient.
\item There are many research questions that stem from this paper. First, it would be interesting to derive posterior consistency and weak convergence results such as those available for the (integer) elliptic problem (\cite{SV13} and \cite{dashti2011uncertainty}). Second, we intend to investigate the computational challenges that arise from the inversion of nonlocal models. Finally, there is applied interest in the learning of variable order models, where the order of the equation is allowed to vary throughout the physical domain. 
\end{itemize}

\section{Appendix}

In this appendix we provide a proof of Lemma \ref{lem:stability} and Proposition \ref{StabilityEigenvalues}.

   \begin{proof}[Proof of Lemma \ref{lem:stability}]
   	Let $g \in \Lavg,$ and let $p:= L_{A}^{-1} g$, $p':= L_{A'}^{-1} g$.  Considering the test function $p-p'$ we deduce that
   	\[    \int_{D} \big\langle A \nabla p, \nabla (p-p')   \big\rangle \,  dx  = \int_{D} g(p-p') \, dx =   \int_{D} \big\langle A'\nabla p' , \nabla (p-p')   \big\rangle  \, dx.   \]
   	Subtracting $\int_{D} \big\langle  A' \nabla p , \nabla (p-p') \big\rangle \, dx$ from both sides of the above equation we obtain 
   	\[ - \int_{D} \big\langle (A-A') \nabla p  ,  \nabla (p - p') \big\rangle \, dx =  \int_{D} \big\langle A' \nabla (p - p' )  ,  \nabla (p - p') \big\rangle \, dx. \]
   	By Cauchy-Schwarz inequality the left hand side of the above expression is bounded above by
   	\[   \lVert A- A' \rVert_\infty   \left( \int_{D} \lvert \nabla p  \rvert^2 dx \right)^{1/2}\left( \int_{D} \lvert \nabla (p - p')\rvert^2 dx \right)^{1/2}.     \]
   	Hence,
   	\begin{align*}
   	\begin{split}
   	\lambda_{A'} \int_D \lvert \nabla (p -p') \rvert^2   dx  & \leq   \int_{D} \big\langle A' \nabla (p - p' )  ,  \nabla (p - p') \big\rangle dx 
   	\\&  \leq  \lVert A- A' \rVert_\infty   \left( \int_{D} \lvert \nabla p  \rvert^2 dx \right)^{1/2}\left( \int_{D} \lvert \nabla (p - p')\rvert^2 dx \right)^{1/2}. 
   	\end{split}
   	\end{align*}
   	We conclude that, 
   	\begin{equation*}
   	\left(\int_D \lvert \nabla (p -p') \rvert^2   dx \right)^{1/2} \leq \frac{1}{\lambda_{A'}}\lVert A- A' \rVert_\infty   \left( \int_{D} \lvert \nabla p  \rvert^2 dx \right)^{1/2}. 
   	\end{equation*}
    The above inequality combined with Poincar\'e inequality and a standard energy argument yield 
   	\begin{equation*}
   	\| p - p' \|_{L^2}  \leq  C \left(\int_D \lvert \nabla (p -p') \rvert^2   dx \right)^{1/2} \leq \frac{C}{\lambda_{A'} \lambda_A}\lVert A- A' \rVert_\infty   \| g \|_{L^2(D)},
   	\end{equation*}
   	where $C$ is a constant only depending on $D$. Therefore, 
   	\[    \|  L^{-1}_A - L^{-1}_{A'} \|_{op} \leq   \frac{C}{\lambda_{A'} \lambda_A}\lVert A- A' \rVert_\infty.      \]
   \end{proof}

Before proving Proposition \ref{StabilityEigenvalues} we show that for arbitrary uniformly elliptic $A$ and $A'$ we have that
\begin{equation}\label{EllipticityConst}
|\lambda_A - \lambda_{A'}| \le \|A- A'\|_\infty.
\end{equation}
For an arbitrary $d\times d$ matrix $B$ we denote by $|B|$ its spectral norm. If $B$ is positive definite, i.e. $B\in \Sp,$ we denote by $\lambda_B^{min}$ its smallest eigenvalue. Then, for $B, B' \in \Sp,$ 
\begin{equation}\label{ineqmineig}
|\lambda_{B}^{min} - \lambda_{B'}^{min}   |  \leq  |  B-B' |. 
\end{equation}
Indeed, by  Courant Fisher theorem, if $x$ is an arbitrary element of $\R^d$ with unit Euclidean norm, 
\begin{align*} 
\lambda_{B}^{min}  &\le \langle Bx, x \rangle =\big\langle (B-B')x, x \big\rangle + \langle B'x, x \rangle \\
&\le |B-B'| +  \langle B'x, x \rangle.
\end{align*}
Minimizing the right-hand side over  $x\in \R^d$ with unit norm,
$$\lambda_{B}^{min}   \le |B-B'| + \lambda_{B'}^{min}.$$
A symmetry argument then gives \eqref{ineqmineig}. It then  follows from \eqref{ineqmineig} that, for a.e. $x \in D,$
\[  \lambda_A \leq   \lambda_ {A(x)}^{min}  \leq  \lambda_{A'(x)}^{min} +  | A(x) - A'(x) | \leq  \lambda_{A'(x)}^{min}  +  \lVert A - A' \rVert_\infty,  \]
and so
\[  \lambda_A \leq \lambda_{A'} +   \lVert A - A' \rVert_\infty .\]
Changing the roles of $A'$ and $A$ we obtain \eqref{EllipticityConst}.

\begin{proof}[Proof of Proposition \ref{StabilityEigenvalues}]
The proof is by induction on $N$.

\textit{Base case $N=1$.} The first eigenvalues of $L_{A}^{-1}$ and $L_{A'}^{-1}$  can be written respectively as
\[   \frac{1}{\lambda_{1}}  = \max_{ \lVert p \rVert_{L^2}=1 }   \langle  L_{A}^{-1}p  ,  p \rangle_{L^2} \]
\[  \frac{1}{\lambda_{1}'}  = \max_{ \lVert p \rVert_{L^2}=1 }   \langle  L_{A'}^{-1}p  ,  p \rangle_{L^2}.  \]
Now, for an arbitrary $p\in \Lavg$ with $\lVert p\rVert_{L^2}=1$, we have
\begin{align}
\begin{split}
 \frac{1}{\lambda_1} \geq \langle  L_A^{-1} p  ,  p \rangle_{L^2}   &= \langle  (L_A^{-1}  - L_{A'}^{-1})p , p  \rangle_{L^2}   + \langle L_{A'}^{-1} p , p \rangle_{L^2} 
 \\ & \geq - \lVert L_A^{-1} - L_{A'}^{-1} \rVert_{op} +   \langle L_{A'}^{-1} p , p \rangle_{L^2} 
 \\& \geq  - \frac{C}{\lambda_A\lambda_{A'}} \lVert A - A' \rVert_{\infty} + \langle L_{A'}^{-1} p , p \rangle_{L^2},
\end{split}
\end{align}
where the last inequality follows from Lemma \ref{lem:stability}. Maximizing over all unit norm $p\in \Lavg$  in the last line of the above expression gives
\[  \frac{1}{\lambda_1'} - \frac{1}{\lambda_1} \leq \frac{C}{\lambda_A\lambda_{A'}} \lVert A - A' \rVert_{\infty} . \]
Reversing the roles of $A$ and $A'$ and combining with the above inequality
\begin{equation*}
  \left \lvert  \frac{1}{\lambda_1'} - \frac{1}{\lambda_1} \right \rvert  \leq  \frac{C}{\lambda_A\lambda_{A'}} \lVert A - A' \rVert_{\infty}. 
\end{equation*}
Therefore, if $\lVert A- A'    \rVert_{\infty}  \leq  \frac{\lambda_A}{2}$ we deduce using \eqref{EllipticityConst} that
\begin{equation}
  \left \lvert  \frac{1}{\lambda_1'} - \frac{1}{\lambda_1} \right \rvert  \leq  C_A \lVert A - A' \rVert_{\infty}, 
  \label{IneqLambdas} 
\end{equation}
for a constant $C_A$ that only depends on $A$.

Let us now focus on proving the statement about eigenfunctions.  Let  $\psi_1'$ be a unit norm eigenfunction of $L^{-1}_{A'}$ with eigenvalue $\frac{1}{\lambda_{1}'}$. Let us denote by 
\[\frac{1}{\tilde{\lambda}_{1}} , \frac{1}{\tilde{\lambda}_{2}}, \dots \] 
the \textit{different} eigenvalues of $L_{A}^{-1}$. Also, denote by $P_i $ the projection onto the eigenspace of $L_{A}^{-1}$ associated to the eigenvalue $\frac{1}{\tilde{\lambda}_i}$. Then, 
\begin{align*}
\begin{split}
  L_{A}^{-1} \psi_1' & = \frac{1}{\lambda_1}   P_1 \psi_1'     + \sum_{i=2}^{\infty} \frac{1}{\tilde{\lambda}_i}   P_i \psi_1',
\end{split}
\end{align*}
so that 
\[    L_{A}^{-1}\psi_1' - \frac{1}{\lambda_1}\psi_1'  = \sum_{i=2}^\infty  \left(  \frac{1}{\tilde{\lambda}_i} - \frac{1}{\lambda_1} \right) P_i \psi_1'.  \] 
Thus, 
\begin{align}
\begin{split}
\Big\lVert  L_{A}^{-1}\psi_1'   - \frac{1}{\lambda_1} \psi_{1}'  \Big\rVert_{L^2}^2 &= \sum_{i=2}^\infty \left( \frac{1}{\tilde{\lambda}_i} - \frac{1}{\lambda_1} \right)^2  \lVert   P_i\psi_1' \rVert_{L^2}^2
\\& \geq \left( \frac{1}{\lambda_1} -\frac{1}{\tilde{\lambda}_2} \right)^2 \sum_{i=2}^\infty \lVert P_i \psi_1'  \rVert^2_{L^2}
\\& = \left( \frac{1}{\lambda_1} -\frac{1}{\tilde{\lambda}_2} \right)^2 \lVert \psi_1' - P_1 \psi_1'  \rVert_{L^2}^2.
\end{split}
\end{align}
Hence, 
\begin{align}
\begin{split}
\lVert  \psi_1' - P_1 \psi_1'    \rVert_{L^2}  & \leq   \frac{\lambda_1 \tilde{\lambda}_2}{\tilde{\lambda}_2 - \lambda_1   }   \Big\lVert  L_{A}^{-1}\psi_1'   - \frac{1}{\lambda_1} \psi_1'  \Big\rVert_{L^2} 
\\& \leq  C_A \left(  \lVert  L_{A}^{-1}\psi_1'   - L_{A'}^{-1} \psi_1'  \rVert_{L^2} + \Big\lVert L_{A'}^{-1} \psi_1' - \frac{1}{\lambda_1}\psi_1'  \Big\rVert_{L^2}       \right) 
\\& \leq  C_A \left(  \lVert  L_{A}^{-1}  - L_{A'}^{-1} \rVert_{op} +   \Big\lvert    \frac{1}{\lambda_1} - \frac{1}{\lambda_1'}    \Big\rvert    \right)
\\& \leq C_A \lVert  A- A'  \rVert_{\infty},
\label{AuxAppendix1}
\end{split}
\end{align}
where the last inequality follows from Lemma \ref{lem:stability} and from \eqref{IneqLambdas} (assuming that $\lVert A' - A \rVert_\infty \leq \frac{\lambda_A}{2}$).  Choosing $\delta_A :=  \min \{ \frac{1}{2C_A},  \frac{\lambda_A}{2}  \}$, where $C_A$ is the constant in the last line of \eqref{AuxAppendix1}, we notice that if $\lVert A - A' \rVert_\infty < \delta_A$, then
\[  \psi_1 := \frac{1}{\lVert P_1 \psi_1'\rVert}  P_1 \psi_1'  \]
is a normalized eigenfunction of $L_{A}^{-1}$ with eigenvalue $\frac{1}{\lambda_1}.$ Moreover,
\[ \lVert \psi_1 - \psi_1' \rVert_{L^2} \leq C_A \lVert  A  - A' \rVert_\infty   \]
for a constant $C_A$ (not necessarily equal to that in \eqref{AuxAppendix1}).

\textit{Inductive step.} Let us now suppose that the result is true for $N-1$, so that there are constants $C_A$ and $\delta_A>0$ such that if $\lVert A - A' \rVert_\infty   \leq \delta_A   $ then 
\[  \Big\lvert   \frac{1}{\lambda_i}   - \frac{1}{\lambda_i'} \Big\rvert  \leq  C_A \lVert  A - A' \rVert_{\infty} ,\quad i=1, \dots, N-1,\]
and
\[ \lVert  \psi_i - \psi_{i}' \rVert_{L^2}  \leq  C_A \lVert  A - A' \rVert_{\infty} ,\quad i=1, \dots, N-1,    \]
for some orthonormal eigenfunctions $\psi_1, \dots, \psi_{N-1}$ of $L_{A}^{-1}$ with eigenvalues $\frac{1}{\lambda_1}, \dots, \frac{1}{\lambda_{N-1}}$, and some orthonormal eigenfunctions $\psi_1', \dots, \psi_{N-1}'$ of $L_{A'}^{-1}$ with eigenvalues $\frac{1}{\lambda_1'}, \dots, \frac{1}{\lambda_{N-1}'}$. 

Let us start with the statement about the eigenvalues. Indeed, $\frac{1}{\lambda_N}$ and $\frac{1}{\lambda_{N}'}$  can be written respectively as
\[  \frac{1}{\lambda_N}  =  \max_{p \in \Sigma^{\perp}, \lVert p \rVert= 1 }   \langle L_{A}^{-1} p, p \rangle_{L^2}   \]
\[  \frac{1}{\lambda_N'}  =  \max_{p \in \Sigma^{'\perp}, \lVert p \rVert= 1 }   \langle L_{A'}^{-1} p, p \rangle_{L^2},   \]
where $ \Sigma :=  \Span \{  \psi_1, \dots, \psi_{N-1} \}$ and $\Sigma' = \Span \{  \psi_1' , \dots, \psi_{N-1}'   \} $, are as above.  For an arbitrary $p' \in \Sigma^{'\perp} $ with $\lVert  p' \rVert_{L^2}=1$, let us consider 
\[  p := p' - \sum_{i=1}^{N-1}\langle  p' , \psi_i     \rangle_{L^2} \psi_i. \]  
 From the induction hypothesis
 \begin{align}
 \begin{split}
\left \lVert   \sum_{i=1}^{N-1}  \frac{1}{\lambda_i}\langle p' , \psi_i \rangle_{L^2} \psi_i  \right\rVert_{L^2}  &=  \left  \lVert   \sum_{i=1}^{N-1}  \frac{1}{\lambda_i}\langle p' , \psi_i \rangle_{L^2} \psi_i  -  \sum_{i=1}^{N-1}  \frac{1}{\lambda_i'}\langle p' , \psi_i' \rangle_{L^2} \psi_i'     \right\rVert_{L^2}  \\ &\leq C_{N,A} \lVert A - A'\rVert_\infty,   
 \label{AuxAppen2}
 \end{split}
 \end{align}
provided $\lVert A' - A \rVert_\infty \leq \delta_A $.  So if  $\lVert A' - A \rVert_\infty \leq \delta_A $, then 
\begin{align}
\begin{split}
\langle L_{A}^{-1} p, p \rangle_{L^2}  &=  \langle L_{A}^{-1}p',  p' \rangle_{L^2}   - \sum_{i=1}^{N-1} \frac{1}{\lambda_i}\langle p', \psi_i \rangle_{L^2}^2
\\ & \geq \langle L_{A}^{-1} p',  p' \rangle_{L^2} - C_{N,A}\lVert  A-A' \rVert_\infty,
\end{split}
\label{AuxAppen3}
\end{align}
where the last inequality follows from \eqref{AuxAppen2}. On the other hand, assuming $ \lVert A - A'\rVert_\infty \leq \frac{1}{2 C_{N,A}}$ in \eqref{AuxAppen2}, we deduce that
\[    \frac{\langle L_A^{-1}p ,p  \rangle_{L^2}}{\langle p, p \rangle_{L^2}} ( 1  + C_{N,A} \lVert A -A' \rVert_\infty )  \geq       \langle L_A^{-1}p ,  p \rangle_{L^2},  \]
and so combining with \eqref{AuxAppen3} and the fact that $\lVert L_{A'}^{-1}p' - L_{A}^{-1}p' \rVert_{L^2} \leq C_{A} \lVert A - A' \rVert _\infty $ we conclude that
\[  \frac{1}{\lambda_N}  + C_{N,A} \lVert A - A' \rVert_\infty  \geq  \langle L_{A'}^{-1}p' ,  p' \rangle_{L^2}.  \]
Taking the maximum over unit norm $p'$
\[ \frac{1}{\lambda_N}  +  C_{N,A} \lVert A - A' \rVert_\infty  \geq \frac{1}{\lambda_N'}.  \]
We may now switch the roles of $A$ and $A'$ and obtain a similar inequality which then implies that
\begin{equation}
 \left \lvert  \frac{1}{\lambda_N} - \frac{1}{\lambda_N'} \right   \rvert \leq C_{N,A} \lVert A - A' \rVert_\infty . 
 \label{AuxAppen4}
\end{equation}

Let us now establish the statement about the eigenfunctions. Let $\psi_N'$ be a unit norm eigenfunction of $L_{A'}^{-1}$ with eigenvalue $\frac{1}{\lambda_N'}$.  Let us denote by 
\[ \frac{1}{\tilde{\lambda}_{k_{N} }  }  ,   \frac{1}{\tilde{\lambda}_{k_{N} + 1}  }, \dots    \]
the \textit{different} eigenvalues of $L_{A}^{-1}$ greater than or equal to $\frac{1}{\lambda_N}$. Also, for $i\geq k_N $ we let $E_{i}$ be the eigenspace of $L_{A}^{-1}$ associated to the eigenvalue $\frac{1}{\tilde{\lambda}_i}$, and denote by $P_i$ the projection onto $E_i$. Finally, we let $P_{-}$ be the projection onto $E_{k_N} \setminus \Span \{\psi_1, \dots, \psi_{N-1} \}$. Then, 
\[  L_{A}^{-1}\psi_N'  := \sum_{i=1}^{N-1} \frac{1}{\lambda_i}  \langle  \psi_N' ,  \psi_i \rangle_{L^2}   \psi_i   + \frac{1}{\lambda_N} P_-\psi_N'  + \sum_{i=k_N + 1} ^\infty \frac{1}{\tilde{\lambda}_i} P_i \psi_N',    \]
and in particular
\[  L_{A}^{-1}\psi_N'  -  \frac{1}{\lambda_N} \psi_{N}'  =  \sum_{i=1}^{N-1} \left( \frac{1}{\lambda_i}- \frac{1}{\lambda_N}\right)  \langle  \psi_N' ,  \psi_i \rangle_{L^2}   \psi_i     + \sum_{i=k_N + 1} ^\infty \left( \frac{1}{\tilde{\lambda}_i}- \frac{1}{\lambda_N} \right) P_i \psi_N'. \]
Using the induction hypothesis and the fact that $\psi_N'$ is orthogonal to all $\psi_1', \dots, \psi_{N-1}'$, we deduce that
\begin{align}
\begin{split}
\Big\lVert L_{A}^{-1}\psi_N'  -  \frac{1}{\lambda_N} \psi_{N}'  \Big\rVert_{L^2}^2 \geq - C_{N,A} \lVert A - A' \rVert_{\infty}^2 + \left( \frac{1}{\tilde{\lambda}_{k_N + 1}}- \frac{1}{\lambda_N} \right)^2 \sum_{i=k_N + 1}^\infty \lVert P_i \psi_N'  \rVert_{L^2}^2.
\end{split}
\end{align}
Combining with \eqref{AuxAppen4}, and arguing as in the base case $N=1$,  it  follows that 
\[   \Big\lVert  \psi_{N}' - P_{-} \psi_{N}'  - \sum_{i=1}^{N-1} \langle \psi_N' , \psi_i  \rangle_{L^2} \psi_i \Big\rVert_{L^2}=\Bigl(\sum_{i=k_N + 1}^\infty \lVert P_i \psi_N'  \rVert_{L^2}^2 \Bigr)^{1/2} \leq  C_{A,N} \lVert A- A' \rVert_{\infty}.\]
Using one more time the induction hypothesis and the fact that $\psi_N'$ is orthogonal to all $\psi_1', \dots, \psi_{N-1}'$,
we deduce that
\[  \lVert  \psi_{N}' - P_{-} \psi_{N}' \rVert_{L^2} \leq  C_{A,N} \lVert A- A' \rVert_{\infty}.\]
As in the base case, we see that provided that $\lVert A - A' \rVert_\infty $ is small enough, then
\[ \lVert \psi_N'  - \psi_{N} \rVert_{L^2}  \leq C_{A,N} \lVert A - A' \rVert_\infty, \]
where $\psi_N := \frac{1}{\lVert P_{-} \psi_{N}' \rVert_{L^2}} P_{-} \psi_{N}' $.

\end{proof}

  \bibliographystyle{plain}
\bibliography{isbib}

\begin{thebibliography}{10}

\bibitem{acosta2015fractional}
G.~L. Acosta and J.~P. Borthagaray.
\newblock {A fractional Laplace equation: regularity of solutions and Finite
  Element approximations}.
\newblock {\em arXiv preprint arXiv:1507.08970}, 2015.

\bibitem{applebaum2009levy}
D.~Applebaum.
\newblock {\em L{\'e}vy processes and stochastic calculus}.
\newblock Cambridge university press, 2009.

\bibitem{atangana2013use}
A.~Atangana and N.~Bildik.
\newblock The use of fractional order derivative to predict the groundwater
  flow.
\newblock {\em Mathematical Problems in Engineering}, 2013, 2013.

\bibitem{bagley1983fractional}
R.~L. Bagley and J.~Torvik.
\newblock Fractional calculus-a different approach to the analysis of
  viscoelastically damped structures.
\newblock {\em AIAA journal}, 21(5):741--748, 1983.

\bibitem{bertoin1998levy}
J.~Bertoin.
\newblock {\em L{\'e}vy processes}, volume 121.
\newblock Cambridge university press, 1998.

\bibitem{bochner1949diffusion}
S.~Bochner.
\newblock Diffusion equation and stochastic processes.
\newblock {\em Proceedings of the National Academy of Sciences},
  35(7):368--370, 1949.

\bibitem{caffarelli2007extension}
L.~Caffarelli and L.~Silvestre.
\newblock {An extension problem related to the fractional Laplacian}.
\newblock {\em Communications in partial differential equations},
  32(8):1245--1260, 2007.

\bibitem{caffarelli2016fractional}
L.~A. Caffarelli and P.~R. Stinga.
\newblock {Fractional elliptic equations, Caccioppoli estimates and
  regularity}.
\newblock In {\em Annales de l'Institut Henri Poincare (C) Non Linear
  Analysis}, volume~33, pages 767--807. Elsevier, 2016.

\bibitem{calvetti2015variable}
D.~Calvetti, E.~Somersalo, and R.~Spies.
\newblock Variable order smoothness priors for ill-posed inverse problems.
\newblock {\em Mathematics of Computation}, 84(294):1753--1773, 2015.

\bibitem{DS15}
M.~Dashti and A.~M. Stuart.
\newblock The bayesian approach to inverse problems.
\newblock Handbook of Uncertainty Quantification.

\bibitem{dashti2011uncertainty}
M.~Dashti and A.~M. Stuart.
\newblock Uncertainty quantification and weak approximation of an elliptic
  inverse problem.
\newblock {\em SIAM Journal on Numerical Analysis}, 49(6):2524--2542, 2011.

\bibitem{de2011fractional}
A.~de~Pablo, F.~Quir{\'o}s, A.~Rodr{\'\i}guez, and J.~L. V{\'a}zquez.
\newblock A fractional porous medium equation.
\newblock {\em Advances in Mathematics}, 226(2):1378--1409, 2011.

\bibitem{di2012hitchhikers}
E.~Di~Nezza, G.~Palatucci, and E.~Valdinoci.
\newblock {Hitchhikerʼs guide to the fractional Sobolev spaces}.
\newblock {\em Bulletin des Sciences Math{\'e}matiques}, 136(5):521--573, 2012.

\bibitem{engl1996regularization}
H.~W. Engl, M.~Hanke, and A.~Neubauer.
\newblock {\em Regularization of inverse problems}, volume 375.
\newblock Springer Science \& Business Media, 1996.

\bibitem{jin2015tutorial}
B.~Jin and W.~Rundell.
\newblock A tutorial on inverse problems for anomalous diffusion processes.
\newblock {\em Inverse Problems}, 31(3):035003, 2015.

\bibitem{jolly2016data}
M.~S. Jolly, V.~R. Martinez, and E.~S. Titi.
\newblock A data assimilation algorithm for the subcritical surface
  quasi-geostrophic equation.
\newblock {\em arXiv preprint arXiv:1607.08574}, 2016.

\bibitem{kaipio2006statistical}
J.~Kaipio and E.~Somersalo.
\newblock {\em Statistical and computational inverse problems}, volume 160.
\newblock Springer Science \& Business Media, 2006.

\bibitem{klafter2005anomalous}
J.~Klafter and I.~G. Sokolov.
\newblock Anomalous diffusion spreads its wings.
\newblock {\em Physics world}, 18(8):29, 2005.

\bibitem{nochetto2015pde}
R.~H. Nochetto, E.~Ot{\'a}rola, and A.~J. Salgado.
\newblock {A PDE approach to fractional diffusion in general domains: a priori
  error analysis}.
\newblock {\em Foundations of Computational Mathematics}, 15(3):733--791, 2015.

\bibitem{sasso2011application}
M.~Sasso, G.~Palmieri, and D.~Amodio.
\newblock Application of fractional derivative models in linear viscoelastic
  problems.
\newblock {\em Mechanics of Time-Dependent Materials}, 15(4):367--387, 2011.

\bibitem{stinga2010fractional}
P.~R. Stinga.
\newblock Fractional powers of second order partial differential operators:
  extension problem and regularity theory.
\newblock 2010.

\bibitem{stinga2010extension}
P.~R. Stinga and J.~L. Torrea.
\newblock {Extension problem and Harnack's inequality for some fractional
  operators}.
\newblock {\em Communications in Partial Differential Equations},
  35(11):2092--2122, 2010.

\bibitem{AS10}
A.~M. Stuart.
\newblock Inverse problems: a {B}ayesian perspective.
\newblock {\em Acta Numerica}, 19:451--559, 2010.

\bibitem{SV13}
S.~J. Vollmer.
\newblock {Posterior consistency for Bayesian inverse problems through
  stability and regression results}.
\newblock {\em Inverse Problems}, 29(12):125011, 2013.

\bibitem{xiu2010numerical}
D.~Xiu.
\newblock {\em Numerical methods for stochastic computations: a spectral method
  approach}.
\newblock Princeton University Press, 2010.

\bibitem{xiu2002wiener}
D.~Xiu and G.~E. Karniadakis.
\newblock {The Wiener--Askey polynomial chaos for stochastic differential
  equations}.
\newblock {\em SIAM journal on scientific computing}, 24(2):619--644, 2002.

\bibitem{zayernouri2015fractional}
M.~Zayernouri and G.~E. Karniadakis.
\newblock {Fractional spectral collocation methods for linear and nonlinear
  variable order FPDEs}.
\newblock {\em Journal of Computational Physics}, 293:312--338, 2015.

\bibitem{zhang2015undetermined}
Z.~Zhang.
\newblock An undetermined coefficient problem for a fractional diffusion
  equation.
\newblock {\em Inverse Problems}, 32(1):015011, 2015.

\end{thebibliography}

\end{document}